\documentclass[
	bibliography=totoc, 
	onecolumn,
	british
]{scrxartcl}
\usepackage{
    xcolor
}
\usepackage[style=authoryear, seconds=true, urldate=iso, backend=biber]{biblatex}
\newcommand\second{2\textsuperscript{nd}}

\newcommand\quor {\quad \lor  \quad}
\newcommand\alter[1]{\p{-1}^{#1}}

\newcommand\spaceas[3][l]{
    \makebox[\widthof{#3}][#1]{#2}
}
\newcommand\spacemath[3][r]{
    \spaceas[#1]{$#2$}{$#3$}
}
\newcommand\spaceby[3][=]{
    \spacemath[c]
    {\by[#1]{#2}}
    {\by[#1]{#3}}
}
\newcommand\empar[2][{}]{%
    \marginpar{
        \raggedright
        \emph{-- \nolinebreak\ifx&#1&{#2}\else{#1}\fi}
    }%
    \emph{#2}%
}
\SmartPairedDelimiter{hopen}[)
\renewcommand\i{\imath}
\renewcommand\j{\jmath}
\defineshorthand{"~}{\babelhyphen{nobreak}}
\useshorthands{"}
\newcommand\textif{\text{if }}
\newcommand\textelse{\text{else}}
\newcommand\even{\mathrm{even}}
\newcommand\odd {\mathrm{odd}}
\newcommand\iseven[1][{}]{
    \text{ is even{#1}}
}
\newcommand\isodd [1][{}]{
    \spacemath[l]
    {\text{ is odd{#1}}}
    {\iseven[#1]{}}
}
\newcommand\abc[1][2]{\mathcal A_{#1}}
\newcommand\zero{\bm 0}
\newcommand\one{\bm 1}
\newcommand\cp[3]{#1 \p{#2 \p #3}}
\newcommand\slice[3]{#1_{\hopen{#2, #3}}}
\newcommand\lsq{l^2}
\newcommand\rsq{r^2}
\newcommand\sft{sft}

\newcommand\crlzero[1]{\crl \zero ^ {#1}}
\newcommand\crlone [1]{\crl \one  ^ {#1}}
\newcommand\crlT   [1]{\crl T     ^ {#1}}
\newcommand\crlCT  [1]{\crl {CT}  ^ {#1}}
\newcommand\graphicheight{%
    0.33\textheight
}
\newcommand\blue[1]{{\color{blue}{#1}}}
\newcommand\red[1]{{\color{red}{#1}}}
\newcommand\green[1]{{\color{green}{#1}}}

\theoremstyle{plain}
\newtheorem{thrm}{Theorem}[section]
\newtheorem{lemma}[thrm]{Lemma}
\newtheorem{corll}[thrm]{Corollary}
\newtheorem{observ}[thrm]{Observation}

\newtheorem{fact}[thrm]{Fact}

\theoremstyle{definition}
\newtheorem*{defn}{Definition}
\newtheorem*{notation}{Notation}

\theoremstyle{remark}
\newtheorem*{remark}{Remark}
\newtheorem*{ex}{Example}
\newtheorem*{countex}{Counterexample}

\numberwithin{equation}{section}

\titlehead{\includegraphics[width=.5\textwidth]{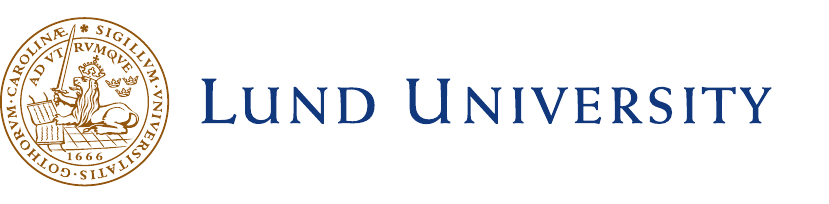}}
\title{An alternating colouring function on strings}
\author{Jonathan Garbe\\Lund University}
\date{2024}
\begin{document}
    \pagenumbering{roman}
    	\maketitle
    	\begin{abstract}

An alternating colouring function is defined on strings over the alphabet $\crl{\zero, \one}$. It divides the strings in colourable and non-colourable ones. The points in the subshift of finite type defined by forbidding all non-colourable strings of a certain length alternate between states of one colour and states of the other colour. In other words, the points in the \second{} power shifts all have the same colour. 

The number $K_n$ of non-colourable strings of length $n \ge 2$ is shown to be $2 \cdot \p{J_{n-2} + 1}$ where $J$ is the sequence of Jacobsthal numbers. The number of sources and sinks in the de Bruijn graph of dimension $n \ge 3$ with non-colourable edges removed is shown each to be $K_n - 4$. 
    	\end{abstract} 
    	\clearpage
        \tableofcontents
        \addsec{Preface}

This paper is an extract from my master thesis (\cite{20garbe}) that I wrote during the spring semester 2020. As many will remember, that was when the \textsc{covid-19} pandemic hit the world. I was lucky that all I had left to complete my major was the thesis while fellow students struggled with improvised online teaching as even Swedish universities closed their doors to mitigate the spread of the virus. 

I started with a research question in the area of dynamical systems that I tried to tackle by looking at de Bruijn sequences. (They are named after Nicolaas Govert de Bruijn who \citeyear{46debruijn} published his famous result regarding their number.) After a time however I noticed some symmetries of de Bruijn graphs that I tried to formalise. In particular, I noticed that by deleting palindromes I could categorise the remaining edges such that in any path an edge of one category would follow an edge of the other. 

In the following weeks I was able to formalise that notion by defining the function~$\psi$. Theorem \ref{thrm:psilllpsirrr} shows it actually has the property I was looking for. It can therefore be seen as the most important result, but it is not surprising as $\psi$ was made to fulfil that property. However, observation \ref{observ:coliffpalin} came as a disappointment: The set of strings that $\psi$ does not assign a colour to is identical with the set of palindromes only up to a string length of $5$ and theorem \ref{thrm:ximne0} shows that the definition of $\psi$ cannot easily be improved. Corollary \ref{corll:Kn2n} instead shows that in the limit $\frac{1}{6}$ of all strings are not colourable. Surprisingly however, the number of non-colourable strings of certain length turns out to be an integer sequence so far unknown even to the On-Line Encyclopedia of Integer Sequences \parencite{oeis} that had been a great help in finding that sequence. 

While for strings of even length the definition of $\psi$ is rather straight-forward, for strings of odd length the definition first requires two other functions $\phi$ and $\xi$. Those two functions feature some remarkable relations. Theorem \ref{thrm:phi0xi0}, which establishes a relation between $\phi$ and $\xi$ being equal to zero, to me is the most surprising result of my master thesis. 

For personal reasons some time has passed between my master thesis and this paper. I beg pardon for the delay. I am grateful for my supervisor Jörg Schmeling for having introduced me into the subject and especially for my examiner Tomas Persson who has put more effort in supporting the production of my thesis and this paper than I would have dared to ask for.\\*[1em]

Cologne, Autumn 2024

\emph{Jonathan Garbe}
    \pagenumbering{arabic}
        \section{Notation}

\begin{notation}
    Define \empar[$l$, $r$, $m$, $R$, $C$, $T$, $CT$]{$l, r : \abc^{\ge 1} \to \abc^*, m : \abc^{\ge 2} \to \abc^*, R, C : \abc^* \to \abc^*, T, CT : \N \to \abc^*$} by
    \begin{align*}
        l \p w &= w_{[0, \# w - 1)}
    \\
        r \p w &= w_{[1, \# w)}
    \\
        m &= l \circ r
    \\ 
        R \p{w_0 \dots w_{\#w-1}} &= w_{\#w-1} \dots w_0
    \\
        C \p{w_0 \dots w_{\#w-1}} 
        &= C \p{w_0} \dots C \p{w_{\#w-1}}
        \text{, where } C \p \zero = \one \land C \p \one = \zero
    \\ 
        T^n &= 
        \begin{cases}
            \upepsilon & \textif n = 0
        \\
            T^{n-1} \zero & \textif n \text{ odd}
        \\
            T^{n-1} \one & \textif n \text{ even} > 0
        \end{cases}            
    \\
        CT^n &= C \p {T^n}.
    \end{align*}
\end{notation}  

\begin{remark}
    $l \p w, r \p w$ are prefix and suffix with length $\# w - 1$ of $w$. $m \p w$ is the proper infix of length $\# w - 2$. $R \p w$ is the reverse of $w$ (meaning that $w$ is a palindrome if and only if $R \p w = w$) while $C \p w$ is its complement. $T^n, CT^n$ are the strings of length $n$ starting with $\zero, \one$ respectively and alternating between $\zero$ and~$\one$.
\end{remark}

\begin{ex}
    \begin{align*}
        l \p{\bm{0010110}} &= \bm{001011}
    \\
        r \p{\bm{0010110}} &= \bm{010110}
    \\ 
        m \p{\bm{0010110}} &= \bm{01011}
    \\         
        R \p{\bm{0010110}} &= \bm{0110100}
    \\
        C \p{\bm{0010110}} &= \bm{1101001}
    \\ 
        T^7 &= \bm{0101010}
    \\ 
        CT^7 &= \bm{1010101}
    \end{align*}
\end{ex}

\begin{observ}
    \begin{align}
        l \circ r &= r \circ l
        \label{eq:lrrl}
    \\ 
        l \circ m = m \circ l &\land 
        r \circ m = m \circ r
        \label{eq:lmmlrmmr}
    \\
        R \circ l = r \circ R &\land 
        R \circ r = l \circ R
        \label{eq:RlrRRrlR}
    \\ 
        R \circ m &= m \circ R
        \label{eq:RmmR}
    \\             
        C \circ l = l \circ C &\land 
        C \circ r = r \circ C
        \label{eq:CllCCrrC}
    \\ 
        C \circ m &= m \circ C
        \label{eq:CmmC}
    \\ 
        R \circ C &= C \circ R
        \label{eq:RCCR}
    \end{align}
\end{observ}

\section{The function \texorpdfstring{$\xi$}{xi}} 
\label{section:xi}

\begin{defn}
    \empar[the function $\xi$]{$\xi$}$: \abc^* \to  \crl{-1, 0, 1}$ is defined recursively by \nopagebreak
    \begin{align*}
        \xi\p w &= 
        \begin{cases}
            0  & \textif w = \upepsilon
        \\
            1  & 
            \textif w \in \crl T ^ {\even \ge 0}
        \\ 
            -1  & 
            \textif w \in \crl {CT} ^ {\even \ge 0}
        \\
            \sgn \p{\xi \p{l \p w} + \xi \p{r \p w}} & \textelse.
        \end{cases}
    \end{align*}
\end{defn}

\begin{table}
    \centering
    \begin{tabular}{c|cccccccc}
        $w$ && 
        $\upepsilon$ & 
        $\zero$ & $\one$ & 
        $\bm{00}$ & $\bm{01}$ & $\bm{10}$ & $\bm{11}$ 
    \\ \hline
        $\xi \p w$ && 
        0 & 
        0 & 0 & 
        0 & 1 & -1 & 0 
    \\ \hline \hline
        $w$ & 
        $\bm{000}$ & $\bm{001}$ & $\bm{010}$ & $\bm{011}$ & $\bm{100}$ & $\bm{101}$ & $\bm{110}$ & $\bm{111}$         
    \\ \hline
        $\xi \p w$ & 
        0 & 1 & 0 & 1 & -1 & 0 & -1 & 0    
    \end{tabular}
    \caption{$\xi \p w$ for $w \in \abc^{\le3}$}
    \label{tab:xi}
\hrulefill\end{table}

\begin{remark}
    The function $\xi$ tells on which side of the vertical axis a vertex appears in a de Bruijn graph when drawn as is done in the examples given in this thesis. In figures~\ref{fig:deBruijn1xi} to \ref{fig:deBruijn4xi} the background is coloured according to the value of $\xi$ evaluated on the vertices. The vertices $w$ for which $\xi \p w = 0$ are found in the centre. The letter $\xi$ was chosen with the word \emph{x-axis} in mind. 
\end{remark}

\begin{figure}
    \centering
    \includegraphics
    [height=\graphicheight]
    {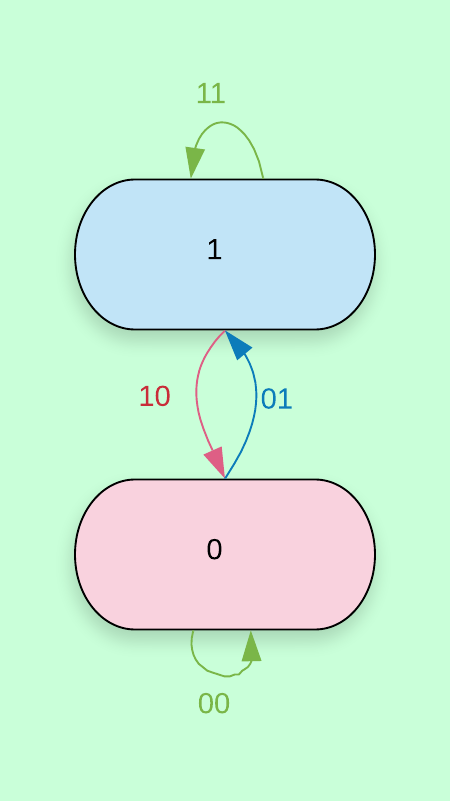}
    \caption{The $1$-dimensional de Bruijn graph. The vertex $\zero$ and the edge $\bm{10}$ are coloured red indicating that $\phi \p \zero = \phi \p{\bm{10}} = -1$ while the vertex $\one$ and the edge $\bm{01}$ are coloured blue indicating $\phi \p \one = \phi \p{\bm{01}} = 1$. The edges $\bm{00}, \bm{11}$ and the background are coloured green because $\xi \p \zero = \xi \p \one = \phi \p{\bm{00}} = \phi \p{\bm{11}} = 0$.}
    \label{fig:deBruijn1xi}
    \hrulefill
\end{figure}
\begin{figure}
    \centering
    \includegraphics
    [height=\textwidth]
    {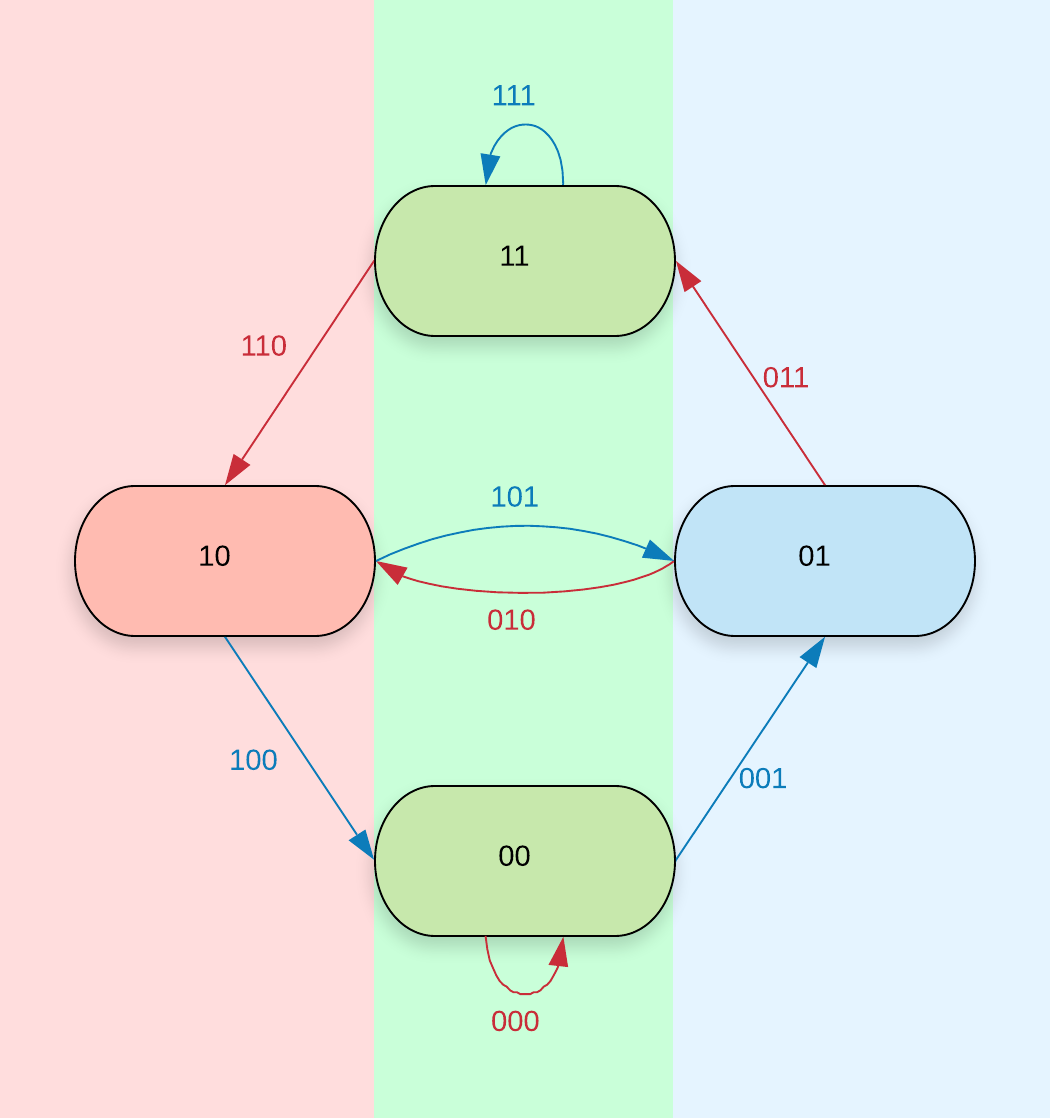}
    \caption{The $2$-dimensional de Bruijn graph. The vertices and edges are coloured red if $\phi = -1$, green if $\phi = 0$ and blue if $\phi = 1$. The background of the vertex $\bm{10}$ is painted red because $\xi \p{\bm{10}} = -1$. In the centre the background is green because $\xi \p{\bm{00}} = \xi \p{\bm{11}} = 0$ and for the vertex $\bm{01}$ it is blue because $\xi \p{\bm{01}} = 1$.}
    \label{fig:deBruijn2xi}
    \hrulefill 
\end{figure}
\begin{figure}
    \centering
    \includegraphics
    [height=\textwidth]
    {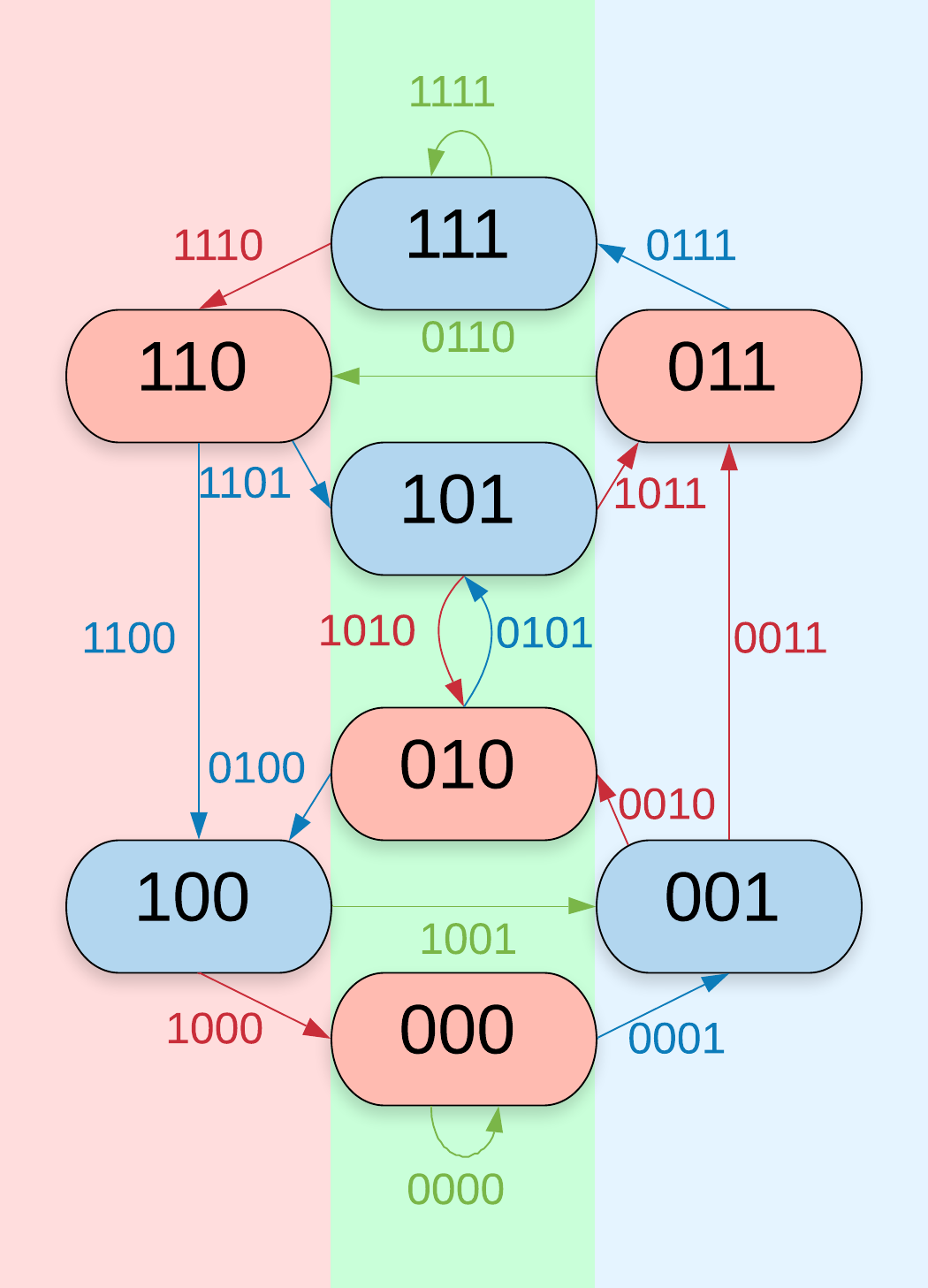}
    \caption{The $3$-dimensional de Bruijn graph. The vertices and edges are coloured red if $\phi = -1$, green if $\phi = 0$ and blue if $\phi = 1$. The background of the vertices is painted red where $\xi = -1$, green where $\xi = 0$ and blue where $\xi = 1$. Note that since the vertices have odd length, none of them is green.}
    \label{fig:deBruijn3xi}
    \hrulefill
\end{figure}
\begin{figure}
    \centering
    \includegraphics
    [height=\textwidth]
    {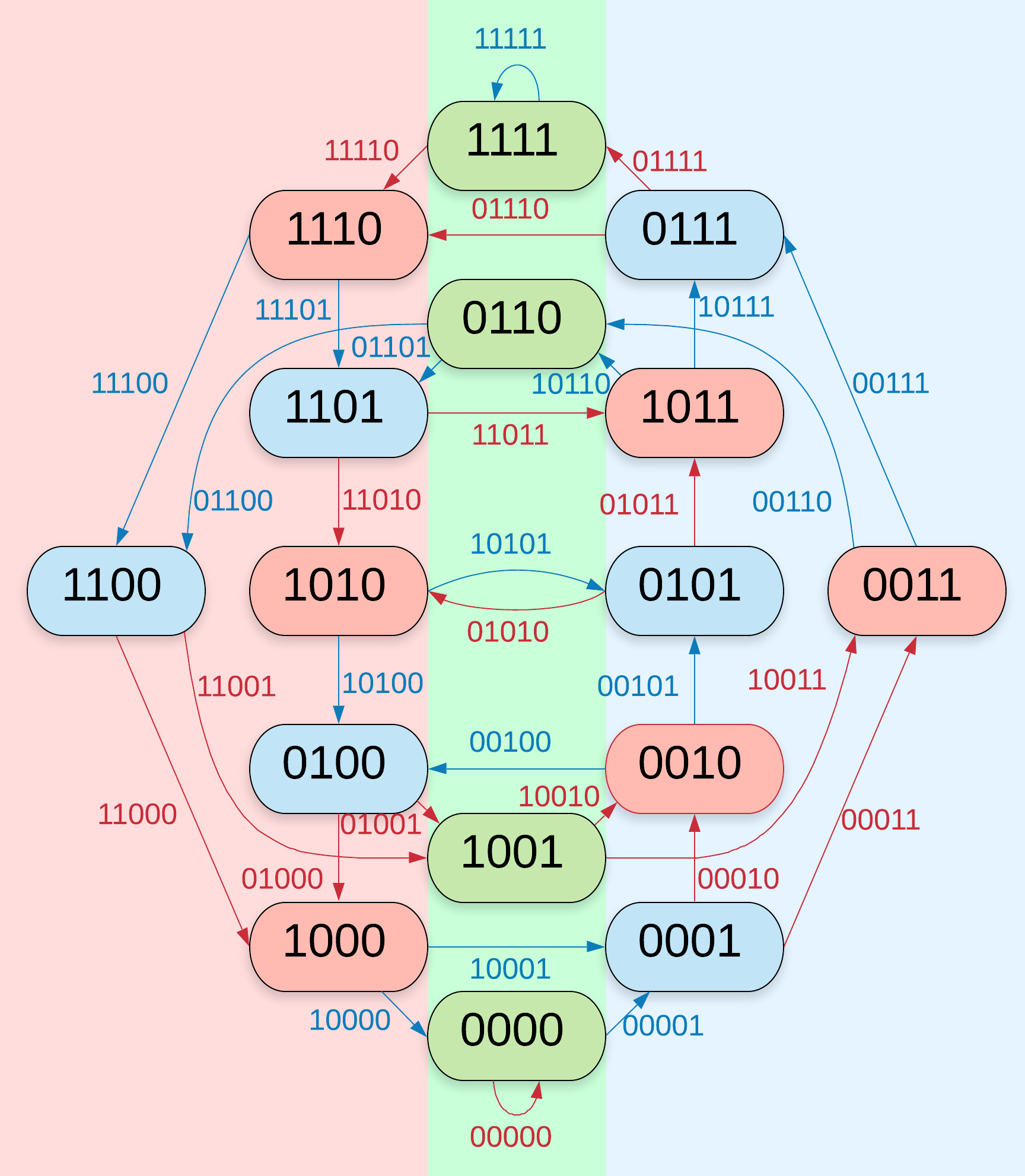}
    \caption{The $4$-dimensional de Bruijn graph. The vertices and edges are coloured red if $\phi = -1$, green if $\phi = 0$ and blue if $\phi = 1$. The background of the vertices is painted red where $\xi = -1$, green where $\xi = 0$ and blue where $\xi = 1$. Note that since the edges have odd length, none of them is green while the green vertices are exactly those that also have green background.}
    \label{fig:deBruijn4xi}
    \hrulefill
\end{figure}

\begin{lemma} \label{lemma:xiCxiR}
    \begin{align}
        \xi \circ C = \xi \circ R = -\xi 
        \label{eq:xiCxiR}
    \end{align}
\end{lemma}
\begin{proof}
    By induction on $\# w$. 
    \begin{align*}
        \xi \p{C \p \upepsilon} &= - \xi \p \upepsilon = \xi \p{R \p \upepsilon} = 0
    \end{align*}
    Now assume, the statement holds for $\# w < n$. Let $\# w = n$. 
    
    If $w \in \crl{T}^\even \cup \crl{CT}^\even$, the statement follows directly from the definition of $\xi$. Otherwise, 
    \begin{align*}
        \begin{split}
            \p{\xi \circ C} \p w 
            \spaceby{}{\eqref{eq:xiCxiR}}&
            \sgn \p{
                \p{\xi \circ l \circ C} \p w + 
                \p{\xi \circ r \circ C} \p w
            } 
        \\ 
            \by{\eqref{eq:CllCCrrC}}&
            \sgn \p{
                \p{\xi \circ C \circ l} \p w + 
                \p{\xi \circ C \circ r} \p w
            } 
        \\ 
            \by{\eqref{eq:xiCxiR}}&
            -\sgn \p{
                \p{\xi \circ l} \p w + 
                \p{\xi \circ r} \p w
            } 
        \\             
           \spaceby{}{\eqref{eq:xiCxiR}}&
           -\xi \p w            
        \end{split}
    \\
        \begin{split}
            \p{\xi \circ R} \p w 
            \spaceby{}{\eqref{eq:xiCxiR}}&
            \sgn \p{
                \p{\xi \circ l \circ R} \p w + 
                \p{\xi \circ r \circ R} \p w
            } 
        \\ 
            \by{\eqref{eq:RlrRRrlR}}&
            \sgn \p{
                \p{\xi \circ R \circ r} \p w + 
                \p{\xi \circ R \circ l} \p w
            } 
        \\ 
            \by{\eqref{eq:xiCxiR}}&
            -\sgn \p{
                \p{\xi \circ r} \p w + 
                \p{\xi \circ l} \p w
            } 
        \\            
           \spaceby{}{\eqref{eq:xiCxiR}}&
           -\xi \p w.          
        \end{split}
    \end{align*}
\end{proof}

\begin{corll} \label{corll:xipalindrome}
    Let $w \in \abc^*$ be a palindrome. Then 
    \begin{align}
        \xi \p w = 0
        \label{eq:xipalindrome}.        
    \end{align}
\end{corll}
\begin{proof}
    Since for a palindrome $w$, $R \p w = w$, equation \eqref{eq:xiCxiR} becomes 
    \begin{align*}
        \xi \p w &= -\xi \p w
    \end{align*}
    which establishes the statement. 
\end{proof}

\begin{countex}
    $\bm{001101}$ is an example for a string that is not a palindrome although $$\xi \p{\bm{001101}} = 0.$$
\end{countex}

\begin{corll} \label{corll:xiTodd}
    Let $k$ be odd. Then 
    \begin{align}
        \xi \p{T^k} = 
        \xi \p{CT^k} = 
        0.
        \label{eq:xiTodd}
    \end{align}
\end{corll}
\begin{proof}
    Since $T^k, CT^k$ are palindromes the statement follows directly from corollary ~\ref{corll:xipalindrome}.
\end{proof}

\begin{observ} \label{observ:xi0} 
    Let $w \in \abc^{\ge 1}$.
    \begin{align}
        \xi \p w = 0 &\iff
        w \notin \crl{T}^\even \cup \crl{CT}^\even \land
        \xi \p{l \p w} = -\xi \p{r \p w}
        \label{eq:xi0}
    \end{align}
\end{observ}

\begin{corll}
    Let $w \in \abc^{\ge 1}$ such that $\xi \p{l \p w}, \xi \p{r \p w} \ne 0$. Then
    \begin{align}
        \xi \p w = 0 \iff
        \xi \p{l \p w} \ne \xi \p{r \p w}
        \label{eq:xi0*}
    \end{align}
\end{corll}
\begin{proof}
    From corollary ~\ref{corll:xiTodd} it follows that $w \notin \crl{T}^\even \cup \crl{CT}^\even$, so the statement follows from observation \ref{observ:xi0}.
\end{proof}

\begin{remark}
    When $\xi \p{l \p w}, \xi \p{r \p w} \ne 0$, it is equivalent to say $\xi \p{l \p w} = -\xi \p{r \p w}$ or $\xi \p{l \p w} \ne \xi \p{r \p w}$. For aesthetic reasons, the latter formulation will be preferred, also in future occasions. 
\end{remark}

\begin{lemma} \label{lemma:xixil} \label{lemma:xixir}
    Let $w \in \abc^{\ge 1}$ such that $\xi \p{w} \ne 0$. Then 
    \begin{align}
        \xi \p{l \p{w}} \ne 0 &\implies
        \xi \p{w} = \xi \p{l \p w}
        \label{eq:xixir}
    \\ 
        \xi \p{r \p{w}} \ne 0 &\implies
        \xi \p{w} = \xi \p{r \p w}
        \label{eq:xixil}.
    \end{align}
\end{lemma}
\begin{proof}
    From $\xi \p{l \p w} \ne 0 \lor \cp \xi r w \ne 0$ it follows that $w \by[\notin]{\eqref{eq:xiTodd}} \crl T ^ \even \cup \crl{CT} ^ \even$, so 
    \begin{align*}
        \xi \p w = 
        \sgn \p{\cp \xi l w + \cp \xi r w}.
    \end{align*}
    Now if one of the summands is non-zero, $\xi \p w$ must either equal that summand or be zero. 
\end{proof}

\begin{lemma} \label{lemma:xi0i1j}
    Let $\i, \j \ge 1$. Then 
    \begin{align}
        \xi \p{\zero^\i \one^\j} = 1.
        \label{eq:xi0i1j}
    \end{align}
\end{lemma}
\begin{proof}
    Induction on $\i + \j$. $\xi \p{\bm{01}} = 1$, which establishes the statement for $\i + \j = 2$. Now assume the statement holds for $\i + \j < n \in \N$. Let $\i + \j = n$. 
    \begin{align*}
        \xi \p{\zero^\i \one^\j} = 
        \sgn \p{
            \xi \p{\zero^\i \one^{\j-1}} + 
            \xi \p{\zero^{\i-1} \one^\j}
        },
    \end{align*}
    where by the induction hypothesis at least one of the summands is 1 while the other cannot be negative either. The statement follows. 
\end{proof}

\begin{corll} \label{corll:xi1j0i}
    \begin{align}
        \xi \p{\one^\j \zero^\i} = -1
        \label{eq:xi1j0i}
    \end{align}
\end{corll}
\begin{proof}
    \begin{align*}
        \xi \p{\one^\j \zero^\i} 
        \by{\eqref{eq:xiCxiR}}
        -\xi \p{\zero^\i \one^\j} 
        \by{\eqref{eq:xi0i1j}}
        -1
    \end{align*}
\end{proof}

\begin{lemma} \label{lemma:xiTk0}
    Let $k \ge 2$. Then 
    \begin{align}
        \xi \p{T^k \zero} &= 
        \begin{cases}
             0 & \textif k \iseven
        \\ 
            -1 & \textif k \isodd
        \end{cases}
        \label{eq:xiTk0}
    \\
        \xi \p{T^k \one} &= 1
        \label{eq:xiTk1}
    \\            
        \xi \p{CT^k \zero} &= -1
        \label{eq:xiCTk0}
    \\
        \xi \p{CT^k \one} &= 
        \begin{cases}
            0 & \textif k \iseven
        \\ 
            1 & \textif k \isodd[.]
        \label{eq:xiCTk1}
        \end{cases}
    \end{align}
\end{lemma}
\begin{proof}
    By induction on $k$. Note first, that the statement holds for $k = 2$. Assume then that it holds for a certain $k \ge 2$. 
    
    Case $k$ is even: 
    \begin{align*}
        \xi \p{T^{k+1} \zero} &= 
        \sgn \p{\xi \p{T^{k+1}} + \xi{\p{CT^k \zero}}}
        \by{\eqref{eq:xiTodd}, \eqref{eq:xiCTk0}}
        \sgn \p{0 - 1} = 
        -1
    \\
        \xi \p{T^{k+1} \one} &= 
        \xi \p{T^{k+2}} = 
        1
    \\
        \xi \p{CT^{k+1} \zero} &= 
        \xi \p{CT^{k+2}} = 
        -1
    \\
        \xi \p{CT^{k+1} \one} &= 
        \sgn \p{\xi \p{CT^{k+1}} + \xi{\p{T^k \one}}}
        \by{\eqref{eq:xiTodd}, \eqref{eq:xiTk1}}
        \sgn \p{0 + 1} = 
        1
    \end{align*}
    
    Case $k$ is odd: 
    \begin{align*}
        \xi \p{T^{k+1} \zero} &= 
        \xi \p{T^{k+2}} 
        \by{\eqref{eq:xiTodd}}
        0
    \\
        \xi \p{T^{k+1} \one} &= 
        \sgn \p{\xi \p{T^{k+1}} + \xi{\p{CT^k \one}}}
        \by{\eqref{eq:xiCTk1}}
        \sgn \p{1 + 1} = 
        1
    \\
        \xi \p{CT^{k+1} \zero} &= 
        \sgn \p{\xi \p{CT^{k+1}} + \xi{\p{T^k \zero}}}
        \by{\eqref{eq:xiTk0}}
        \sgn \p{-1 - 1} = 
        -1
    \\
        \xi \p{CT^{k+1} \one} &= 
        \xi \p{CT^{k+2}} 
        \by{\eqref{eq:xiTodd}}
        0
    \end{align*}
\end{proof}

\begin{lemma} \label{lemma:xilr0}
    \label{lemma:inequalityproof}
    Let $w \in \abc^{\ge 1}$ such that $\xi \p{l \p w} = \xi \p{r \p w} = 0$. Then 
    \begin{align}
        \xi \p w = \spacemath 0 {-1} &\implies 
        w \in \crl \zero ^* \cup \crl \one ^*
        \label{eq:xilr0xi0}
    \\ 
        \xi \p w = \spacemath 1 {-1} &\implies 
        w \in \crl{ T}^\even
        \label{eq:xilr0xi1}
    \\ 
        \xi \p w = -1 &\implies 
        w \in \crl{CT}^\even  
        \label{eq:xilr0xim1}.
    \end{align}
\end{lemma}
\begin{proof}
    Induction on $\# w$. For $\# w \le 2$ there are only finitely many cases to consider. 
    Now assume the statements holds for $\# w < n$. Let $\# w = n \ge 3 \land \xi \p{l \p w} = \xi \p{r \p w} = 0$, which means that by observation \ref{observ:xi0}
    \begin{align*}
        \xi \p{l^2 \p w} = 
        - \xi \p{m \p w} = 
        \xi \p{r^2 \p w}.
    \end{align*}
    If $m \p w \in \crl{T}^\even \cup \crl{CT}^\even$, the statement follows from lemma~\ref{lemma:xiTk0}. Otherwise, the above equation can be rewritten the following way:
    \begin{align*}
        \begin{split}
            &\sgn \p{
                \p{\xi \circ l^3} \p w + 
                \p{\xi \circ l^2 \circ r} \p w
            } 
        \\ 
            =  - &\sgn \p{
                \p{\xi \circ l^2 \circ r} \p w +
                \p{\xi \circ l \circ r^2} \p w
            } 
        \\
            = \phantom{-} &\sgn \p{
                \p{\xi \circ l \circ r^2} \p w + 
                \p{\xi \circ r^3} \p w
            }            
        \end{split}
    \end{align*}
    Some playing with the minus sign yields:
    \begin{align*}
        \begin{split}
            &\sgn \p{
                \p{\xi \circ l^3} \p w - 
                \p{- \xi \circ l^2 \circ r} \p w
            } 
        \\
            = &\sgn \p{
                \p{- \xi \circ l^2 \circ r} \p w -
                \p{\xi \circ l \circ r^2} \p w
            } 
        \\
            = &\sgn \p{
                \p{\xi \circ l \circ r^2} \p w -
                \p{- \xi \circ r^3} \p w
            }     
        \end{split}
    \end{align*}
    Hence, one of the following three statements must hold: 
    \begin{align*}
        \p{ \xi \circ l^3} \p w < 
        \p{-\xi \circ l^2 \circ r} \p w &< 
        \p{ \xi \circ l \circ r^2} \p w < 
        \p{-\xi \circ r^3} \p w
    \\ 
        \lor \quad
        \p{ \xi \circ l^3} \p w =
        \p{-\xi \circ l^2 \circ r} \p w &=
        \p{ \xi \circ l \circ r^2} \p w =
        \p{-\xi \circ r^3} \p w
    \\ 
        \lor \quad
        \p{ \xi \circ l^3} \p w >
        \p{-\xi \circ l^2 \circ r} \p w &>
        \p{ \xi \circ l \circ r^2} \p w >
        \p{-\xi \circ r^3} \p w
    \end{align*}
    However, since $\img \xi = \crl{-1, 0, 1}$, the strict inequalities cannot hold, meaning that the equality does. Then 
    \begin{align*}
        \xi \p{l^2 \p w} = 
        \xi \p{m   \p w} = 
        \xi \p{r^2 \p w} = 
        0,     
    \end{align*}
    so by the induction hypothesis 
    \begin{align*}
       l \p w, r \p w \in 
        \crl \zero ^* \cup \crl \one ^*
    \end{align*}
    and since $\# w \ge 3$, 
    \begin{align*}
        w \in \crl \zero ^* \cup \crl \one ^*.
    \end{align*}
\end{proof}
\begin{remark}
    The argument used in this proof, considering $l^3, l^2 \circ r, l \circ r^2$ and $r^3$ of a certain string and using the fact that the functions $\xi$ and $\phi$ only attain three different values will be important also for the proof of lemma~\ref{lemma:philwrw0}. 
\end{remark}

\begin{lemma} \label{lemma:xi0xim0}
    Let $w \in \abc^{\ge 2}$ such that $\xi \p w = 0$. Then 
    \begin{align}
        \xi \p{m \p w} = 0
        \label{eq:xi0xim0}.
    \end{align} 
\end{lemma}
\begin{proof}
    If $m \p w \in \crlT \odd \cup \crlCT \odd$ the statement follows from corollary~\ref{corll:xiTodd} so assume it is not.
    $\xi \p w = 0$ by observation \ref{observ:xi0} implies that $\xi \p{l \p w} = -\xi \p{r \p w}$. By lemma~\ref{lemma:xilr0} either $w \in \crl{T}^\even \cup \crl{CT}^\even$ or $w \in \crl \zero ^* \cup \crl \one ^*$ or $\xi \p{l \p w} = -\xi \p{r \p w} \ne 0$. However, the first case can be excluded since $\xi \p w = 0$. In the second case, $\xi \p{m \p w} \by{\eqref{eq:xipalindrome}} 0$. In the third case 
    \begin{align*}
        \sgn \p{
            \xi \p{l^2 \p w} + 
            \xi \p{m   \p w}
        } =
        -\sgn \p{
            \xi \p{m   \p w} +
            \xi \p{r^2 \p w}
        } \ne 
        0,
    \end{align*}
    which can be reformulated to
    \begin{align*}
        \sgn \p{
            \xi \p{m \p w} - 
            \p{-\xi \p{l^2 \p w}}
        } =
        \sgn \p{
            \p{-\xi \p{r^2 \p w}} - 
            \xi \p{m \p w}
        } \ne 
        0,
    \end{align*}
    so 
    \begin{align*}
        -\xi \p{l^2 \p w} <
        \xi \p{m \p w} < 
        -\xi \p{r^2 \p w}
        \lor 
        -\xi \p{l^2 \p w} >
        \xi \p{m \p w} > 
        -\xi \p{r^2 \p w}.
    \end{align*}
    It follows that $\xi \p{m \p w} = 0$.
\end{proof}

\refstepcounter{thrm} \label{corll:skipped}

\begin{remark}
    The number \ref{corll:skipped} is omitted in numbering the results to keep section~\ref{section:xi} and \ref{section:phi} having the same structure. There is nothing equivalent to say about the function $\xi$ as is said about the function $\phi$ in corollary~\ref{corll:phiodd}.
\end{remark}

\begin{thrm} \label{thrm:ximne0}
    Let $w \in \abc^{\ge 2}$ such that $\xi \p{l^2 \p w} = \xi \p{r^2 \p w} \ne 0$. Then 
    \begin{align}
        \xi \p{m \p w} \ne 0.
        \label{eq:xillxirrxim0}
    \end{align}
\end{thrm}
\begin{proof}
    First of all, note that $\# w \ge 4$, as otherwise $\xi \p{l^2 \p w} = 0$. For $\# w =\nolinebreak 4$, $w \in \crl{T^4, CT^4}$, so $\xi \p{m \p w} \ne 0$.
    Assume now that $\# w \ge 5$ and start with considering the case $l^2 \p w \in \crl T ^\even$. Then $m \p w \in \crl{CT^{\# w -3} \zero, CT^{\# w -3} \one}$, so $\xi \p{m \p w} \by[\ne]{\eqref{eq:xiCTk0}, \eqref{eq:xiCTk1}} 0$. Similarly, the cases $l^2 \p w \in \crl{CT}^\even \lor r^2 \p w \in \crl{T}^\even \cup \crl{CT}^\even$ can be considered. 
    In all other cases, assume $\cp \xi m w = 0$. Then 
    \begin{align}
        \p{ \xi \circ l^2 \circ r} \p w 
        \by{\eqref{eq:xi0}}
        \p{-\xi \circ l \circ r^2} \p w
        \label{eq:proof:xil2rxilr2},
    \end{align}
    so
    \begin{align*}
        \begin{split}
            \cp \xi {l^2} w = \cp \xi {r^2} w \ne 0 
            \iff &
        \\
            \sgn \p{
                \p{\xi \circ l^3} \p w + 
                \p{\xi \circ l^2 \circ r} \p w
            } = &
            \sgn \p{
                \p{\xi \circ l \circ r^2} \p w + 
                \p{\xi \circ r^3} \p w
            } \ne 0
        \\
            \by[\iff]{\eqref{eq:proof:xil2rxilr2}} &
        \\
            \sgn \p{
                \p{\xi \circ l^3} \p w + 
                \p{\xi \circ l^2 \circ r} \p w
            } = &
            \sgn \p{
                \p{\xi \circ r^3} \p w -
                \p{\xi \circ l^2 \circ r} \p w 
            } \ne 0
        \\ 
            \implies &
            \p{\xi \circ l^2 \circ r} \p w = 0.
        \end{split}
    \end{align*}
    From $\p{\xi \circ m} \p w = \p{\xi \circ l^2 \circ r} \p w = \p{\xi \circ l \circ r^2} \p w = 0$ it follows by lemma~\ref{lemma:xilr0} that $m \p w \in \crl \zero ^* \cup \crl \one ^*$. However, then 
    \begin{align*}
        \cp \xi {l^2} w \le 0 \le \cp \xi {r^2} w,
    \end{align*}
    a contradiction. 
\end{proof}

\begin{remark}
    One could try to redefine $\xi$ in an attempt to shrink the subset of $\abc^*$ on which $\xi$ is $0$. However theorem~\ref{thrm:ximne0} shows that that is not easily possible: Already now $\xi$ is zero only for vertices in the de Bruijn graph that have preceding and succeeding vertices with different $\xi$-value. 
\end{remark}

\begin{lemma} \label{lemma:ximw0w}
    Let $n \ge 2, w \in \abc^n$ such that $\cp \xi m w = 0$. Then 
    \begin{align}
        &\cp \xi l w = 
        \cp \xi r w = 
        \spacemath 0 {-1} \iff 
        w \in \crl{\zero^n, \one^n} 
        \label{eq:ximw0w0n1n}
    \\ 
        &\cp \xi l w = 
        \cp \xi r w = 
        \spacemath 1 {-1} \iff 
        w \in \crl{
            \zero T^{n-1},
            T^{n-1} \one
        } \land 
        n \isodd
        \label{eq:ximw0w0Tn1Tn11}
    \\         
        &\cp \xi l w = \cp \xi r w = -1 \iff 
        w \in \crl{
            CT^{n-1} \zero,
            \one CT^{n-1}
        } \land 
        n \isodd
        \label{eq:ximw0wCTn101CTn1}
    \\ 
        &\cp \xi l w = \spacemath 0 {-1} \land 
        \cp \xi r w = \spacemath 1 {-1} \iff 
        w = \zero^{n-1} \one
        \label{eq:ximw0w0n11}
    \\ 
        &\cp \xi l w = \spacemath 0 {-1} \land 
        \cp \xi r w = -1 \iff 
        w = \one^{n-1} \zero
        \label{eq:ximw0w1n10}
    \\ 
        &\cp \xi l w = \spacemath 1 {-1} \land 
        \cp \xi r w = \spacemath 0 {-1} \iff 
        w = \zero \one^{n-1}
        \label{eq:ximw0w01n1}
    \\ 
        &\cp \xi l w = -1 \land 
        \cp \xi r w = \spacemath 0 {-1} \iff 
        w = \one \zero^{n-1}
        \label{eq:ximw0w10n1}.
    \end{align}
\end{lemma}
\begin{proof}
    \
    \begin{description}
        \item [
            \eqref{eq:ximw0w0n1n}, \eqref{eq:ximw0w0n11}, \eqref{eq:ximw0w1n10}
        ]
        From $\cp \xi l w = \cp \xi m w = 0$ it follows that $$\cp \xi {l^2} w = 0.$$ Hence by lemma~\ref{lemma:xilr0} $l \p w \in \crl{\zero^{n-1}, \one^{n-1}}$
        \item [
            \eqref{eq:ximw0w01n1}, \eqref{eq:ximw0w10n1}
        ]
        Similarly one finds in these cases that $r \p w \in \crl{\zero^{n-1}, \one^{n-1}}$. \item [
            \eqref{eq:ximw0w0Tn1Tn11}, 
            \eqref{eq:ximw0wCTn101CTn1}
        ]
        Assume $m \p w \notin \crlT \odd \cup \crlCT \odd$. Then $$\cp \xi {l^2} w = \cp \xi {r^2} w \ne 0$$ which according to theorem~\ref{thrm:ximne0} implies that $m \p w \ne 0$, a contradiction. Also, $w$ must not be a palindrome, leaving exactly these four cases. 
    \end{description}
\end{proof}

\begin{remark}
    Already at this point, $\# \crl{w \in \abc^n; \xi \p w = 0}$ could be investigated. However, that result shall be postponed until section \ref{section:Numbernoncolourable}, as the relevance of that set and its size will be more obvious by then. 
\end{remark}

\begin{lemma} \label{lemma:xilximnegxir}
    Let $w \in \abc^{\ge 4}$ such that $\cp \xi l w \ne \cp \xi m w = -1 \ne \cp \xi r w$. Then $w \in \crlT \even$. 
\end{lemma}
\begin{proof}
    Note that $\cp \xi l w = \cp \xi r w = 0$ as anything else would contradict lemma~\ref{lemma:xixil}. By lemma~\ref{lemma:xilr0} $$w \in \crl \zero ^* \cup \crl \one ^* \cup \crl T ^\even \cup \crl{CT}^\even,$$ but only $w \in \crl T ^\even$ satisfies $\cp \xi m w = -1$. 
\end{proof}

\begin{thrm} \label{thrm:existsxineg}
    Let $k, n \in \N$ and $w \in \abc^{k+n}$ be such that $$\slice w 0 n = \slice w k {k+n} \land \exists \i < k \quad \xi \p{\slice w \i {\i+n}} = 1.$$ Then 
    \begin{align*}
        \exists \j < k \quad 
        \xi \p{\slice w \j {\j+n}} = -1.
    \end{align*}
\end{thrm}
\begin{proof}
    Induction on $n$. As it suffices to consider cycles in the $n$"~dimensional de Bruijn graph (instead of any closed walk $w$ represents), there are  only finitely many cases to consider for $n \in \crl{0, 1, 2}$. Now assume the statement holds for a certain $n$.     
    Let $w \in \abc^{k+n+1}$ such that $$\slice w 0 {n+1} = \slice w k {k+n+1} \land \exists \i < k \quad \xi \p{\slice w \i {\i+n+1}} = 1.$$     
    If $\slice w \i {\i+n+2} \in \crl T ^\odd$ then $\xi \p{\slice w {\i+1} {\i+n+2}} = -1$, which would already establish the statement and if $\slice w \i {\i+n+2} = T^{n+1} \one$ then $\xi \p{\slice w {\i + 2} {\i + n + 2}} = 1$. Otherwise $\slice w \i {\i+n+1} \notin \crlT \even$ so either $\xi \p{\slice w \i {\i+n}} = 1 \lor \xi \p{\slice w {\i+1} {\i+n+1}} = 1$. In any case either the statement is already established or there is a substring of length $n$ to which $\xi$ assigns the value $1$. 
    
    Define $\tilde w$ by replacing any occurrence of $T^{n+2}$ in $w$ by $T^n$, meaning that $T^{n+2}$ is not a substring of $\tilde w$. Set $\tilde k = \# \tilde w - \p{n + 1}$. 
    Note that $$\crl{\slice{\tilde w}{\i}{\i+n}; \i < \tilde k} \subseteq \crl{\slice w \i {\i+n}; \i < k}$$ (in words: any substring of length $n$ of $\tilde w$ is also a substring of $w$), however $$\crl{\slice w \i {\i+n}; \i < k} \subseteq \crl{\slice{\tilde w}{\i}{\i+n}; \i < \tilde k} \cup \crl{CT^n}.$$ (In words: the only substring of length $n$ that has been removed by constructing $\tilde w$ from $w$ is $CT^n$.) 
    Since $\xi \p{CT^n} \ne 1$, it follows that still $\exists \i < \tilde k \quad \xi \p{\slice{\tilde w}{\i}{\i+n}} = 1$. This also shows that $\tilde k \ge 1$ and hence $\# \tilde w \ge n + 2$. 
    
    First, consider the case $\slice {\tilde w} 0 n \ne \slice{\tilde w}{\tilde k}{\tilde k + n}$. That implies $\slice{\tilde w}{\tilde k}{\tilde k + n + 1} \ne \slice w k {k + n + 1}$, which is possible only if $\slice w {k-1} {k+n+1} = T^{n+2}$, so $\slice w 0 {n+1} = \slice w k {k+n+1} = CT^{n+1}$ and hence $\slice{\tilde w}{1}{n+1} = \slice{\tilde w}{\tilde k + 1}{\tilde k + n + 1} = T^n$. By the induction hypothesis, 
    \begin{align}
        \exists \j \quad 
        1 \le \j < \tilde k + 1 \land 
        \xi \p{\slice{\tilde w}{\j}{\j+n}} = 
        -1
        \label{eq:proof:existsj}.
    \end{align}    
    If $\slice{\tilde w}{0}{n} = \slice{\tilde w}{\tilde k}{\tilde k + n}$, the induction hypothesis directly gives $
        \exists \j < \tilde k \quad 
        \xi \p{\slice{\tilde w}{\j}{\j + n}} = 
        -1
    $. If $\j = 0$, one can also pick $\j = \tilde k$, so also in this case equation \eqref{eq:proof:existsj} holds. 
    
    Then by lemma~\ref{lemma:xilximnegxir} 
    \begin{align*}
        \xi \p{\slice{\tilde w}{\j-1}{\j+n}} = -1 \lor 
        \xi \p{\slice{\tilde w}{\j}{\j+n+1}} = -1 \lor 
        \slice{\tilde w}{\j-1}{\j+n+1} = T^{n+2},
    \end{align*}
    where the third option is excluded by the definition of $\tilde w$. The statement now follows from the fact that 
    \begin{align*}
        \crl{
            \slice{\tilde w}{\i}{\i+n+1};
            \i < \tilde k
        } \subseteq
        \crl{
            \slice{w}{\i}{\i+n+1};
            \i < k
        }.
    \end{align*}
    (In words: any substring of length $n + 1$ of $\tilde w$ is also a substring of $w$.)
\end{proof}

\begin{remark}
    The statement also holds when switching the rolls of $1$ and $-1$; the proof is analogous. 
    
    Using the notion of $\xi$ as defining a left, a centre and a right part of the de Bruijn graph, theorem~\ref{thrm:existsxineg} tells that there are no non-empty closed walks in the right part. Figure~\ref{fig:deBruijn4xi} does not contain non-trivial closed walks in the blue (or red) highlighted area. 
    
    The proof idea is simpler than it seems: A string representing a walk in a certain de Bruijn graph also represents a walk in the de Bruijn graph one dimension lower. If a walk passes a vertex $v$ for which $\xi \p v = -1$ of a graph by the construction of the $\xi$ function it will also do so in a higher dimensional de Bruijn graph -- with one notable exception: The vertex $CT^n$ for even $n$. Take for example the walk $\bm{0101}$ in the $2$-dimensional de Bruijn graph. Glancing at figure~\ref{fig:deBruijn2xi} shows that this walk passes both the area with blue and the area with red background. Now consider the same string a walk through the $3$-dimensional de Bruijn graph. A look at figure~\ref{fig:deBruijn3xi} tells that the graph does not pass the area with red background. This one exception is what makes the construction of $\tilde w$ necessary. 
\end{remark}

\begin{corll}
    Let $n \in \N$. The \sft{} generated by prohibiting the strings in $$\crl{w \in \abc^n; \xi \p w \ne 1}$$ is empty. 
\end{corll}
\begin{proof}
    Assume, there were a point $u$ in the \sft. Pick $\i < \j \in \N$ such that $\slice u \i {\i+n} = \slice u \j {\j+n}$. By definition of the \sft, $\xi \p{\slice u \i {\i+n}} = 1$. Set $w = \slice u \i {\j+n}$ and $k = \j - \i$. Then theorem~\ref{thrm:existsxineg} contradicts the definition of the \sft. 
\end{proof}

\section{The function \texorpdfstring {$\phi$} {phi}} 
\label{section:phi}

\begin{remark}
    Section~\ref{section:phi} has the same structure as the previous one. Each result but corollary~\ref{corll:phiodd} can be compared with the result of the same number in section~\ref{section:xi}. Often the proofs are similar. 
\end{remark}

\begin{defn}
    \empar[the function $\phi$]{$\phi$}$: \abc^* \to  \crl{-1, 0, 1}$ is defined recursively by
    \begin{align*}
        \phi\p w &= 
        \begin{cases}
            0  & \textif w = \upepsilon
        \\
            -1  & 
            \textif w \in \crl \zero ^ \odd
        \\ 
            1 & 
            \textif w \in \crl \one ^ \odd
        \\
            \sgn \p{\phi \p{r \p w} - \phi \p{l \p w}} & \textelse.
        \end{cases}
    \end{align*}
\end{defn}

\begin{remark}
    In figures~\ref{fig:deBruijn1xi} to \ref{fig:deBruijn4xi} the vertices and edges are coloured according to the value $\phi$ assigns to them. 
\end{remark}

\begin{table}
    \centering
    \begin{tabular}{c|cccccccc}
        $w$ && 
        $\upepsilon$ & 
        $\zero$ & $\one$ & 
        $\bm{00}$ & $\bm{01}$ & $\bm{10}$ & $\bm{11}$
    \\ \hline
        $\phi \p w$ && 
        0 & 
        -1 & 1 & 
        0 & 1 & -1 & 0 
    \\ \hline \hline
        $w$ &
        $\bm{000}$ & $\bm{001}$ & $\bm{010}$ & $\bm{011}$ & $\bm{100}$ & $\bm{101}$ & $\bm{110}$ & $\bm{111}$         
    \\ \hline
        $\phi \p w$ &
        -1 & 1 & -1 & -1 & 1 & 1 & -1 & 1
    \end{tabular}
    \caption{$\phi \p w$ for $w \in \abc^{\le3}$}
    \label{tab:phi}
    \hrulefill
\end{table}

\begin{lemma}
    \begin{align}
        \phi \circ C &= -\phi &
        \p{\phi \circ R} \p w&= 
        \p{-1}^{\# w + 1} \cdot \phi \p w 
        \label{eq:phiCphiR}
    \end{align}
\end{lemma}
\begin{proof}
    By induction on $\# w$. 
    \begin{align*}
        \phi \p{C \p \upepsilon} &= -\phi \p \upepsilon = \phi \p{R \p \upepsilon} =
        \p{-1}^{0 + 1} \cdot \phi \p \upepsilon = 0
    \end{align*}
    Now assume, the statement holds for $\# w < n$. Let $\# w = n$. 
    
    If $w \in \crl \zero ^\odd \cup \crl \one ^\odd$, the statement follows directly from the definition of $\phi$. Otherwise, 
    \begin{align*}
        \begin{split}
            \p{\phi \circ C} \p w &= 
            \sgn \p{
                \p{\phi \circ r \circ C} \p w - 
                \p{\phi \circ l \circ C} \p w
            } 
        \\ 
            &= \sgn \p{
                \p{\phi \circ C \circ r} \p w - 
                \p{\phi \circ C \circ l} \p w
            } 
        \\ 
            &= -\sgn \p{
                \p{\phi \circ r} \p w - 
                \p{\phi \circ l} \p w
            } 
        \\ 
            &= -\phi \p w            
        \end{split}
    \\
        \begin{split}
            \p{\phi \circ R} \p w &= 
            \sgn \p{
                \p{\phi \circ r \circ R} \p w - 
                \p{\phi \circ l \circ R} \p w
            } 
        \\ 
            &= \sgn \p{
                \p{\phi \circ R \circ l} \p w - 
                \p{\phi \circ R \circ r} \p w
            } 
        \\ 
            &= \p{-1}^n \cdot \sgn \p{
                \p{\phi \circ l} \p w - 
                \p{\phi \circ r} \p w
            } 
        \\ 
            &= \p{-1}^{n+1} \cdot \phi \p w        
        \end{split}
    \end{align*}
\end{proof}

\begin{corll} \label{corll:phipalindrome}
    Let $w \in \abc^\even$ be a palindrome. Then 
    \begin{align*}
        \phi \p w = 0.       
    \end{align*}
\end{corll}
\begin{proof}
    For a palindrome $w$ of even length, equation \eqref{eq:phiCphiR} becomes 
    \begin{align*}
        \phi \p w &= 
        -\phi \p w,
    \end{align*}
    which establishes the statement. 
\end{proof}

\begin{corll} \label{corll:phi01even}
    Let $k$ be even. Then 
    \begin{align}
        \phi \p{\zero^k} = 
        \phi \p{\one^k} = 
        0
        \label{eq:phi01even}
    \end{align}
\end{corll}
\begin{proof}
    Since $\zero^k, \one^k$ are palindromes, the statement follows directly from corollary ~\ref{corll:phipalindrome}.
\end{proof}

\begin{observ} \label{observ:phi0}
    Let $w \in \abc^{\ge 1}$.
    \begin{align}
        \phi \p w = 0 &\iff 
        w \notin \crl \zero ^\odd \cup \crl \one ^\odd \land 
        \phi \p{l \p w} = \phi \p{r \p w}
        \label{eq:phi0}
    \end{align}
\end{observ}

\begin{corll}
    Let $w \in \abc^{\ge 1}$ such that $\phi \p{l \p w}, \phi \p{r \p w} \ne 0$. Then
    \begin{align}
        \phi \p w = 0 \iff 
        \phi \p{l \p w} = \phi \p{r \p w}
        \label{eq:phi0*}
    \end{align}
\end{corll}
\begin{proof}
    From corollary ~\ref{corll:phi01even} it follows that $w \notin \crl \zero ^\odd \cup \crl \one ^\odd$, so the statement follows from observation \ref{observ:xi0}.
\end{proof}

\begin{lemma} \label{lemma:phiphir}
    Let $w \in \abc^{\ge 1}$ such that $\phi \p{w} \ne 0$. Then 
    \begin{align}
        \phi \p{r \p{w}} \ne 0 &\implies
        \phi \p{w} = \phi \p{r \p w}
        \label{eq:phiphir}
    \\ 
        \phi \p{l \p{w}} \ne 0 &\implies
        \phi \p{w} \ne \phi \p{l \p w}
        \label{eq:phiphil}.
    \end{align}
\end{lemma}
\begin{proof}
    From $\phi \p{r \p w} \ne 0 \lor \cp \phi l w \ne 0$ it follows that $w \notin \crl \zero ^ \odd \cup \crl \one ^ \odd$, so \begin{align*}
        \phi \p w = 
        \sgn \p{\cp \phi r w + \p{-\cp \phi l w}}.
    \end{align*}
    Now if one of the summands is non-zero $\phi \p w$ must either equal that summand or be zero. 
\end{proof}

\begin{lemma} \label{lemma:phi0i1j}
    Let $\i, \j \ge 1$. Then 
    \begin{align}
        \phi \p{\zero^\i \one^\j} = \alter{\j+1}.
        \label{eq:phi0i1j}
    \end{align}
\end{lemma}
\begin{proof}
    Induction on $\i + \j$. $\phi \p{\zero \one} = 1$ which establishes the statement for $\i + \j = 2$. Now assume the statement holds for $\i + \j < n \in \N$. Let $\i + \j = n$. Then 
    \begin{align*}
        \phi \p{\zero^\i \one^\j} = 
        \sgn \p{
            \phi \p{\zero^{\i - 1} \one^\j} -
            \phi \p{\zero^\i \one^{\j - 1}}
        },
    \end{align*}
    where
    \begin{align*}
        \j \iseven &\implies
        \phi \p{\zero^{\i - 1} \one^\j} \le 0 \land 
        \phi \p{\zero^\i \one^{\j - 1}} = 1
    \\
        \j \isodd &\implies
        \phi \p{\zero^{\i - 1} \one^\j} = 1 \land 
        \phi \p{\zero^\i \one^{\j - 1}} \le 0.
    \end{align*}
    In both cases the statement holds. 
\end{proof}

\begin{corll} \label{corll:phi1j0i}
    \begin{align*}
        \phi \p{\one^\j \zero^\i} = \alter \i
    \end{align*}
\end{corll}
\begin{proof}
    \begin{align*}
        \phi \p{\one^\j \zero^\i} 
        \by{\eqref{eq:phiCphiR}}
        \alter{\i + \j + 1} \cdot 
        \phi \p{\zero^\i \one^\j} 
        \by{\eqref{eq:phi0i1j}}
        \alter{\i + \j + 1 + \j + 1} = 
        \alter \i
    \end{align*}
\end{proof}

\begin{lemma}
    Let $k > 0$. Then $\phi \p{T^k} = \p{-1}^k$. 
\end{lemma}
\begin{proof}
    Induction on $k$. $\phi \p{T^1} = \phi \p \zero = -1$. Now assume the statement holds for a certain $k$. Then 
    \begin{align*}
        \phi \p{T^{k+1}} = 
        \sgn \p{\phi \p{CT^k} - \phi \p{T^k}} 
        \by{\eqref{eq:phiCphiR}}
        \sgn \p{-2 \cdot \phi \p{T^k}}  = 
        - \phi \p{T^k} = 
        \alter{k+1}.
    \end{align*}
\end{proof}

\begin{lemma} \label{lemma:philwrw0}
    $$
        \forall w \in \abc^{\ge 1} \quad 
        \phi \p{l \p w} = \phi \p{r \p w} = 0 \implies
        w \in \crl \zero ^\odd \cup \crl \one ^\odd
    $$
\end{lemma}
\begin{proof}
    Induction on $\# w$. For $\# w \le 2$ there are only finitely many cases. Now assume the statement holds for $\# w < n$. Let $\# w = n \ge 3$ and $\phi \p{l \p w} = \phi \p{r \p w} = 0$, which means that 
    \begin{align*}
        \phi \p{l^2 \p w} = 
        \phi \p{m   \p w} = 
        \phi \p{r^2 \p w}.
    \end{align*}
    By the argument used in the proof of lemma~\ref{lemma:inequalityproof}, one gets that in fact
    \begin{align*}
        \phi \p{l^2 \p w} = 
        \phi \p{m   \p w} = 
        \phi \p{r^2 \p w} = 
        0,
    \end{align*}
    so by the induction hypothesis, 
    \begin{align*}
        l \p w, r \p w \in 
        \crl \zero ^* \cup \crl \one ^*
    \end{align*}
    and since $\# w \ge 3$, 
    \begin{align*}
        w \in \crl \zero ^* \cup \crl \one ^*.
    \end{align*}
    $\# w$ must be odd because otherwise $\phi \p{l \p w} \ne 0$.
\end{proof}

\begin{lemma} \label{lemma:phimw0}
    Let $w \in \abc^{\ge 2}$ such that $\phi \p w = 0$. Then 
    \begin{align}
        \phi \p{m \p w} = 0
        \label{eq:phimw0}.
    \end{align} 
\end{lemma}
\begin{proof}
    If $m \p w \in \crlzero \even \cup \crlone \even$ the statement follows from corollary~\ref{corll:phi01even} so assume that it is not. 
    $\phi \p w = 0$ implies that $\phi \p{l \p w} = \phi \p{r \p w}$. By lemma~\ref{lemma:philwrw0} either $w \in \crl \zero ^\odd \cup \crl \one ^\odd$ or $\phi \p{l \p w} = \phi \p{r \p w} \ne 0$. However, the first case can be excluded since $\phi \p w = 0$. Hence 
    \begin{align*}
        \sgn \p{
            \phi \p{m   \p w} - 
            \phi \p{l^2 \p w}
        } =
        \sgn \p{
            \phi \p{r^2 \p w} - 
            \phi \p{m   \p w}
        } \ne 
        0,
    \end{align*}
    so 
    \begin{align*}
        \phi \p{l^2 \p w} <
        \phi \p{m   \p w} < 
        \phi \p{r^2 \p w}
        \quor 
        \phi \p{l^2 \p w} >
        \phi \p{m   \p w} > 
        \phi \p{r^2 \p w}.
    \end{align*}
    In both cases it follows that $\phi \p{m \p w} = 0$.
\end{proof}

\begin{corll} \label{corll:phiodd}
    Let $w \in \abc^{\mathrm{odd}}$. Then 
    \begin{align}
        \phi \p w \ne 0
        \label{eq:phiodd}.
    \end{align} 
\end{corll}
\begin{proof}
    Induction on $\# w$. For $\# w = 1$ there are only finitely many cases to consider. Assume now, the statement is true for $\# w < n$. Let $\# w = n$ and $\phi \p w = 0$. Then $\phi \p{m \p w} = 0$, so $n - 2$ is odd. 
\end{proof}

\begin{lemma} \label{lemma:phimne0}
    Let $w \in \abc^{\ge 2}$ such that $\phi \p{l^2 \p w} = \phi \p{r^2 \p w} \ne 0$. Then 
    \begin{align*}
        \phi \p{m \p w} \ne 0.
    \end{align*}
\end{lemma}
\begin{proof}
    $\# w = 2 \implies \cp \phi \lsq w = 0$. Now let $\# w \ge 3$ and $\cp \phi m w = 0$. Then 
    \begin{align}
        \p{\phi \circ l^2 \circ r} \p w 
        \by{\eqref{eq:phi0}}
        \p{\phi \circ l \circ r^2} \p w
        \label{eq:proof:phil2rphilr2}
    \end{align}
    and by corollary~\ref{corll:phiodd} $\# w$ is even so $\lsq \p w, \rsq \p w \notin \crlzero \odd \cup \crlone \odd$. 
    Hence 
    \begin{align*}
        \begin{split}
            \cp \phi {l^2} w = \cp \phi {r^2} w \ne 0 
            \iff &
        \\
            \sgn \p{
                \p{\phi \circ l^2 \circ r} \p w - 
                \p{\phi \circ l^3} \p w 
            } = &
            \sgn \p{
                \p{\phi \circ r^3} \p w - 
                \p{\phi \circ l \circ r^2} \p w 
            } \ne 0
        \\
            \by[\iff]{\eqref{eq:proof:phil2rphilr2}} &
        \\
            \sgn \p{
                \p{-\phi \circ l^3} \p w +
                \p{\phi \circ l^2 \circ r} \p w 
            } = &
            \sgn \p{
                \p{\phi \circ r^3} \p w - 
                \p{\phi \circ l^2 \circ r} \p w 
            } \ne 0
        \\ 
            \implies &
            \p{\phi \circ l^2 \circ r} \p w = 0,
        \end{split}
    \end{align*}
    which is impossible because $\# \p{\p{\lsq \circ r} \p w}$ is odd. 
\end{proof}

\begin{lemma} \label{lemma:phimw0w}
    Let $w \in \abc^2$ such that $\cp \phi m w = 0$. Then 
    \begin{align*}
        \phi \p w \ne 0 \iff 
        w \in \Lambda^{\ge 2} = 
        \bigcup_{n \ge 1} \crl{
            \zero^n \one, 
            \zero \one^n, 
            \one \zero^n, 
            \one^n \zero
        }.
    \end{align*}
\end{lemma}
\begin{proof}
    $\impliedby$ is established in lemma~\ref{lemma:phi0i1j} and corollary~\ref{corll:phi1j0i}. To show $\implies$ assume $w \notin \Lambda^{\ge 2}$ but $\cp \phi m w = 0$. Then $l \p w, r \p w \notin \crlzero \odd \cup \crlone \odd$ and by corollary~\ref{corll:phiodd} $\# w$ is even, so neither is $w$. Hence 
    \begin{align*}
        \cp \phi l w &= 
        \sgn \p{\cp \phi m w - \cp \phi \lsq w} = 
        - \cp \phi \lsq w
    \\
        \cp \phi r w &= 
        \sgn \p{\cp \phi \rsq w - \cp \phi m w} = 
        \cp \phi \rsq w
    \\
        \phi \p w &= 
        \sgn \p{\cp \phi r w - \cp \phi l w} = 
        \sgn \p{\cp \phi \rsq w + \cp \phi \lsq w}.
    \end{align*}
    Since $\cp \# l w$ and $\cp \# r w$ are odd, both summands are non-zero. However, by lemma~\ref{lemma:phimne0} they cannot be equal either. Hence $\phi \p w = 0$. 
\end{proof}

\section{Relations between \texorpdfstring{$\xi$}{xi} and \texorpdfstring{$\phi$}{phi}} 

\begin{lemma} \label{lemma:phimwximw0}
    Let $w \in \abc^{\ge 2}$ such that $\phi \p{m \p w} = \xi \p{m \p w} = 0$. Then 
        \begin{align}
            \phi \p w = \xi \p w
            \label{eq:phimwximw0}.
        \end{align}
\end{lemma}
\begin{proof}
    By corollary~\ref{corll:phiodd} $\# w$ is even. Then by lemmata~\ref{lemma:ximw0w} and \ref{lemma:phimw0w} 
    \begin{align*}
        \xi  \p w \ne 0 \iff 
        \phi \p w \ne 0 \iff 
        w \in \Lambda^{\ge 2}.
    \end{align*}
    Set $k = \# w - 1$. According to lemmata~\ref{lemma:xi0i1j}, \ref{lemma:phi0i1j}, corollaries~\ref{corll:xi1j0i} and \ref{corll:phi1j0i} 
    \begin{align*}
        \xi  \p{\zero^k \one} &= 
        \xi  \p{\zero \one^k} = 
        \phi \p{\zero^k \one} = 
        \phi \p{\zero \one^k} = 
        1 
    \\
        \xi  \p{\one^k \zero} &= 
        \xi  \p{\one \zero^k} = 
        \phi \p{\one^k \zero} = 
        \phi \p{\one \zero^k} = 
        -1.
    \end{align*}
\end{proof}

\begin{thrm} \label{thrm:phi0xi0}
    Let $w \in \abc^*$. 
    \begin{align}
        \phi \p w = 0 \iff 
        \# w \iseven \land \xi \p w = 0
        \label{eq:phi0xi0}
    \end{align}
\end{thrm}
\begin{proof}
    Induction on $\# w$. For $\# w \le 1$ there are only finitely many cases to consider. Now assume, the statement holds for $\# w < n$. Let $\# w = n$. 
    \begin{align*}
        \phi \p w = 0 
        \by[\implies]{\eqref{eq:phiodd}}&
        \# w \iseven
    \\
        \begin{split}
            \phi \p w = 0 
            \by[\implies]{\eqref{eq:phimw0}}&
            \phi \p{m \p w} = 0
        \\
            \by[\implies]{\eqref{eq:phi0xi0}}&
            \xi  \p{m \p w} = 0
        \\
            \by[\implies]{\eqref{eq:phimwximw0}}&
            \xi \p w = \phi \p w = 0          
        \end{split}
    \\ 
        \begin{split}
            \# w \iseven \land \xi \p w = 0
            \by[\implies]{\eqref{eq:xi0xim0}}&
            \# w \iseven \land \xi \p{m \p w} = 0
        \\
            \by[\implies]{\eqref{eq:phi0xi0}}&
            \phi \p{m \p w} = 0
        \\
            \by[\implies]{\eqref{eq:phimwximw0}}&
            \phi \p w = \xi \p w = 0            
        \end{split}
    \end{align*}
\end{proof}

\begin{corll} \label{corll:philphir}
    Let $w \in \abc^{\ge 1}$ be such that $\xi \p{l \p w}, \xi \p{r \p w} \ne 0$. 
    Then 
    \begin{align}
        \phi \p{l \p w} = \phi \p{r \p w} \iff 
        \# w \iseven \land 
        \xi \p{l \p w} \ne \xi \p{r \p w}
        \label{eq:philphir}
    \end{align}
\end{corll}
\begin{proof}
    \begin{align*}
        \phi \p{l \p w} = \phi \p{r \p w} 
        &\by[\iff]{\eqref{eq:phi0*}}
        \phi \p w = 0
    \\
        &\by[\iff]{\eqref{eq:phi0xi0}}
        \# w \iseven \land \xi \p w = 0
    \\
        &\by[\iff]{\eqref{eq:xi0*}}
        \# w \iseven \land 
        \xi \p{l \p w} \ne \xi \p{r \p w}
    \end{align*}
\end{proof}

\begin{remark}
    By giving a necessary and sufficient condition for the exceptions, corollary ~\ref{corll:philphir} shows that usually $\phi \p{l \p w} \ne \phi \p{r \p w}$. Later the alternating colouring function $\psi$ shall be introduced as an improvement of $\phi$ for which the inequality always hold. 
\end{remark}

\begin{lemma} \label{lemma:phillphirr}
    Let $w \in \abc^{\ge 2}$ such that 
    \begin{align}
        \cp \xi {l^2} w, \cp \xi {r^2} w \ne 
        0 = 
        \cp \xi m w
        \label{eq:cond:phillphirrxim0}.
    \end{align}
    Then 
    \begin{align}
        \cp \phi \lsq w = 
        \cp \phi \rsq w \iff
        \# w \isodd
        \label{eq:phillphirr}.
    \end{align}
\end{lemma}
\begin{proof}
    Note that $
        l \p w, r \p w
        \by[\notin]
        {\eqref{eq:cond:phillphirrxim0}} 
        \crlT \even \cup 
        \crlCT \even \cup 
        \crlzero * \cup 
        \crlone *
    $, so 
    \begin{align}
        \cp \xi l w 
        =& 
        \sgn \p{
            \cp \xi {l^2} w + 
            \cp \xi m w
        } 
        \by{\eqref{eq:cond:phillphirrxim0}}
        \cp \xi {l^2} w
        \label{eq:proof:phillphirrxil}
    \\
        \cp \xi r w 
        =& 
        \sgn \p{
            \cp \xi m w + 
            \cp \xi {r^2} w
        } 
        \by{\eqref{eq:cond:phillphirrxim0}}
        \cp \xi {r^2} w
        \label{eq:proof:phillphirrxir}
    \\
        \cp \phi l w 
        =& 
        \sgn \p{
            \cp \phi m w -
            \cp \phi {l^2} w
        } 
        \by{\eqref{eq:cond:phillphirrxim0}}
        -\cp \phi {l^2} w
        \label{eq:proof:phillphirrphil}
    \\
        \cp \phi r w 
        =& 
        \sgn \p{
            \cp \phi {r^2} w -
            \cp \phi m w
        } 
        \by{\eqref{eq:cond:phillphirrxim0}}
        \cp \phi {r^2} w
        \label{eq:proof:phillphirrphir}.
    \end{align}
    Hence 
    \begin{align*}
        \cp \phi {l^2} w = \cp \phi {r^2} w 
        \by[\iff]{
            \eqref{eq:proof:phillphirrphil},
            \eqref{eq:proof:phillphirrphir}
        }&
        \cp \phi l w \ne \cp \phi r w 
    \\  
        \spaceby[\iff]
        {\eqref{eq:philphir}}
        {
            \eqref{eq:proof:phillphirrxil},
            \eqref{eq:proof:phillphirrxir}
        }&
        \# w \isodd \lor 
        \cp \xi l w = \cp \xi r w
    \\ 
        \by[\iff]{
            \eqref{eq:proof:phillphirrxil},
            \eqref{eq:proof:phillphirrxir}
        }&
        \# w \isodd \lor 
        \cp \xi \lsq w = \cp \xi \rsq w
    \\ 
        \spaceby[\iff]
        {\eqref{eq:xillxirrxim0}}
        {
            \eqref{eq:proof:phillphirrxil},
            \eqref{eq:proof:phillphirrxir}
        }&
        \# w \isodd[.]      
    \end{align*}
\end{proof}

\begin{lemma} \label{lemma:philllphirrr}
    Let $k, n \in \N, w \in \abc^\N$ such that $k \ge 1$ and 
    \begin{align*}
        \forall \j \le k \quad \p{
            \xi \p{\slice w \j {\j+n}} \ne 0 \iff
            \j \in \crl{0, k}
        }.
    \end{align*}
    Then 
    \begin{align}
        \phi \p{\slice w 0 n} = 
        \phi \p{\slice w k {k+n}} \iff 
        n \isodd \land 
        \xi \p{\slice w 0 n} \ne 
        \xi \p{\slice w k {k+n}}
        \label{eq:philllphirrr}.
    \end{align}
\end{lemma}
\begin{proof}
    Since $\xi \p{\slice w 0 n} \ne 0$, actually $n \ge 2$. The cases $k \in \crl{1, 2}$ are covered by corollary~\ref{corll:philphir} and lemma~\ref{lemma:phillphirr} respectively. Now let $k \ge 3$. Then by lemma~\ref{lemma:xilr0} 
    \begin{align*}
        m \p w \in \crl{
            \zero^{k+n-2}, \one^{k+n-2}
        } \lor 
        n \isodd \land 
        m \p w \in \crl{
            T^{k+n-2}, CT^{k+n-2}
        }
    \end{align*} 
    Taking into account the condition $\xi \p{\slice w 0 n}, \xi \p{\slice w k {k+n}} \ne 0$ gives 
    \begin{align*}        
        w &\in \crl{
            \one \zero^{k+n-2} \one, 
            \zero \one^{k+n-2} \zero
        } 
    \\
        {} \lor 
        w &\in \crl{
            \zero T^{k+n-2} \zero, 
            \one CT^{k+n-2} \one
        } \land{}            
        k \iseven{} \land{}
        n \isodd 
    \\
        {} \lor
        w &\in \crl{
            \zero T^{k+n-2} \one, 
            \one CT^{k+n-2} \zero
        } \land{}
        k \isodd \land{}
        n \isodd[.]
    \end{align*}
    In each of these cases the statement holds. 
\end{proof}

\begin{remark} 
    The proof idea is to determine those spots in the de Bruijn graph where several edges $w$ for which $\xi \p w = 0$ follow upon each other. By looking at the graphs one gets the impression that happens only around $\zero^n$, $\one^n$ and around the centre of odd-dimensional de Bruijn graphs, which lemma~\ref{lemma:xilr0} confirms. In all other cases the statement is already established by corollary~\ref{corll:philphir} and lemma~\ref{lemma:phillphirr}. 
\end{remark}

\section{The alternating colouring function \texorpdfstring{$\psi$}{psi}}
\label{section:psi}

\begin{defn}
    Define $\psi: \abc^* \to \crl{-1, 0, 1}$ by
    \begin{align}        
        \psi \p w 
        &= \p{\xi \p w}^{\# w} \cdot \phi \p w
        \label{eq:psi}.
    \end{align}
    $\psi$ is called the \empar[alternating colouring function~$\psi$]{alternating colouring function}. 
    
    $w \in \abc^*$ is called \empar[colourable string]{colourable} if $\psi \p w \ne 0$. 
    For $n \ge 2$ the \sft{} generated by prohibiting the strings in $\crl{w \in \abc^n; \psi \p w = 0}$ is called the \empar{maximal $n$"~colourable shift}. 
    Any subshift of the maximal $n$"~colourable shift is called \empar[$n$-colourable shift]{$n$"~colourable}. 
\end{defn}

\begin{remark}
    In figures~\ref{fig:deBruijn1psi} to \ref{fig:deBruijn4psi} the vertices and edges are coloured according to the value $\psi$ assigns to them. Colourable are those vertices and edges that are coloured blue or red. 
    
    The maximal $n$-colourable shift is the edge shift of the $\p{n-1}$-dimensional de Bruijn graph of which all the non-colourable edges (the green ones in figures~\ref{fig:deBruijn1xi} to \ref{fig:deBruijn4xi} that is) have been removed. Equivalently it is the vertex shift of the $n$-dimensional de Bruijn graph of which all non-colourable vertices have been removed. 
\end{remark}

\begin{table}
    \centering
    \begin{tabular}{c|cccccccc}
        $w$ && 
        $\upepsilon$ & 
        $\zero$ & $\one$ & 
        $\bm{00}$ & $\bm{01}$ & $\bm{10}$ & $\bm{11}$        
    \\ \hline
        $\psi \p w$ && 
        0 & 
        0 & 0 & 
        0 & 1 & -1 & 0 
    \\ \hline \hline 
        $w$ & 
        $\bm{000}$ & $\bm{001}$ & $\bm{010}$ & $\bm{011}$ & $\bm{100}$ & $\bm{101}$ & $\bm{110}$ & $\bm{111}$         
    \\ \hline
        $\psi \p w$ & 
        0 & 1 & 0 & -1 & -1 & 0 & 1 & 0
    \end{tabular}
    \caption{The values $\psi \p w$ of the alternating colouring for $w \in \abc^{\le3}$}
    \label{tab:psi}
\hrulefill\end{table}

\begin{figure}
    \centering
    \includegraphics
    [height=\graphicheight]
    {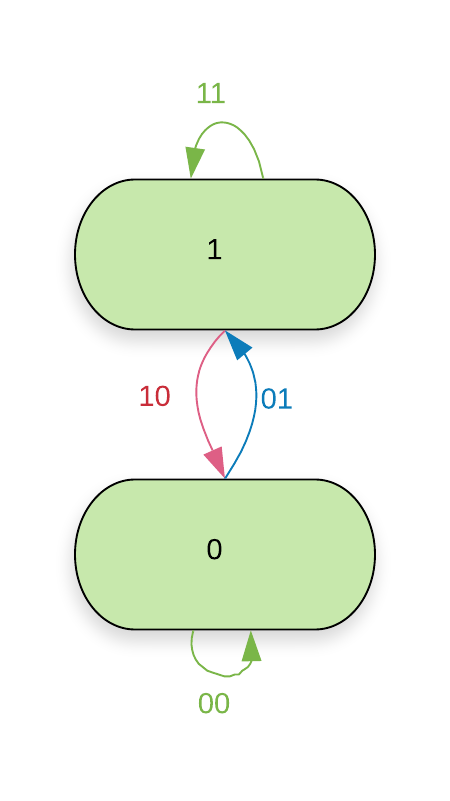}
    \caption{The $1$-dimensional de Bruijn graph. The edges $\bm{01}$ and $\bm{10}$ are coloured blue and red respectively indicating that $\psi \p{\bm{01}} = 1 \land \psi \p{\bm{10}} = -1$. The other two edges and the vertices are coloured green because $\psi$ is $0$ there.}
    \label{fig:deBruijn1psi}
     \hrulefill  \end{figure}
\begin{figure}
    \centering
    \includegraphics
    [height=\graphicheight]
    {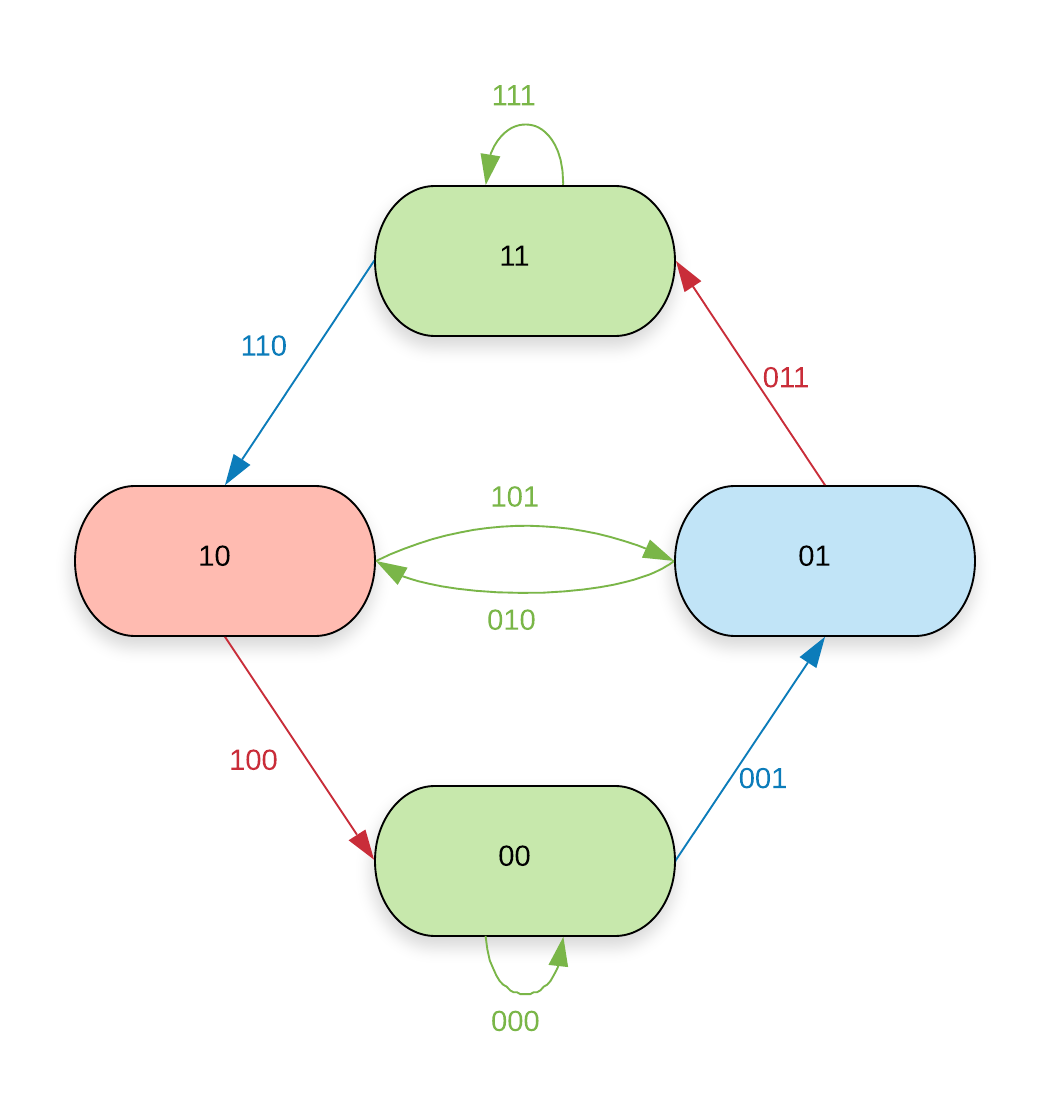}
    \caption{The $2$-dimensional de Bruijn graph. The vertices and edges are coloured red if $\psi = -1$, green if $\psi = 0$ and blue if $\psi = 1$.}
    \label{fig:deBruijn2psi}
\hrulefill\end{figure}
\begin{figure}
    \centering
    \includegraphics
    [height=\graphicheight]
    {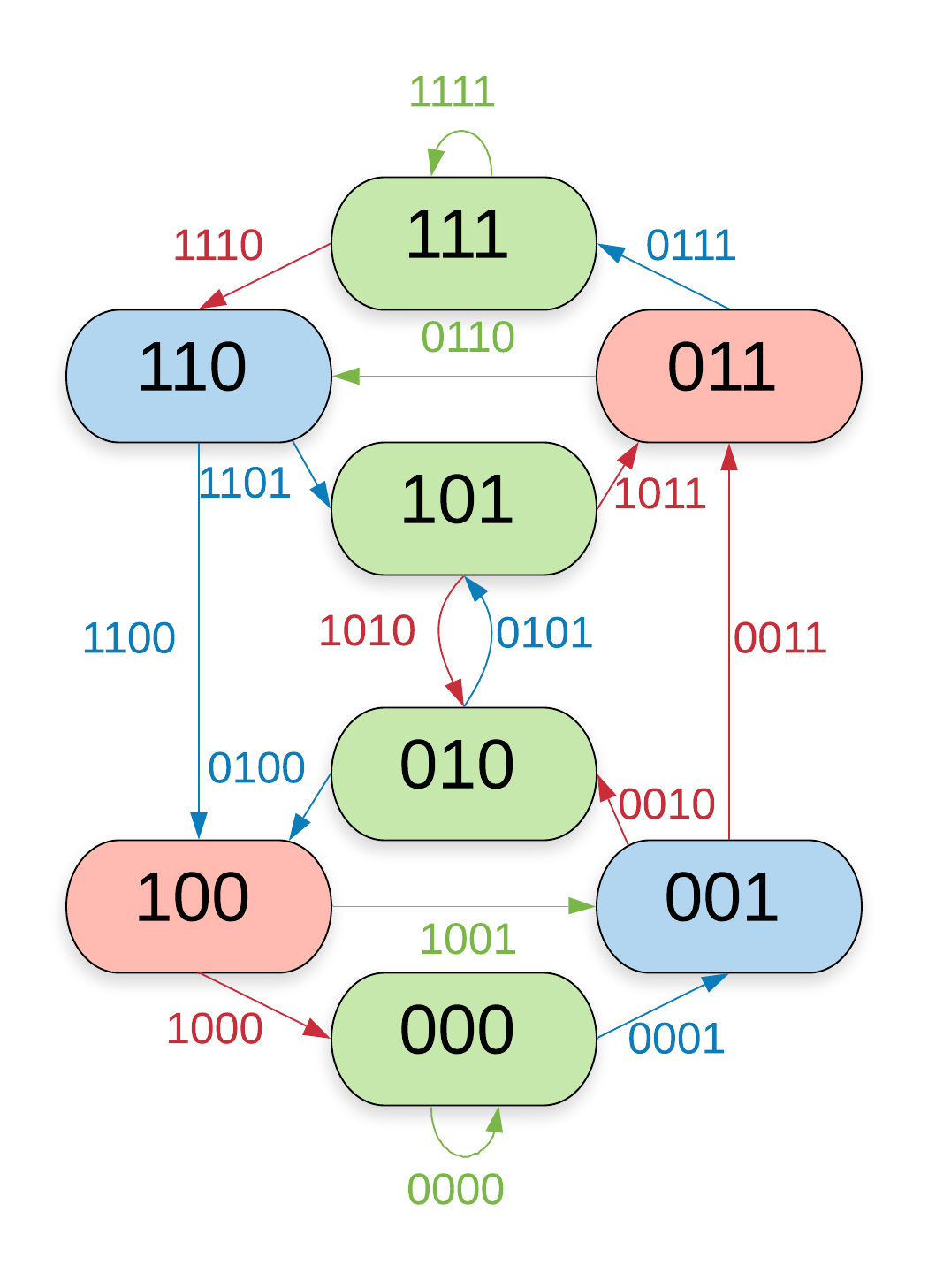}
    \caption{The $3$-dimensional de Bruijn graph. The vertices and edges are coloured red if $\psi = -1$, green if $\psi = 0$ and blue if $\psi = 1$.}
    \label{fig:deBruijn3psi}
\hrulefill\end{figure}
\begin{figure}
    \centering
    \includegraphics
    [width=\textwidth]
    {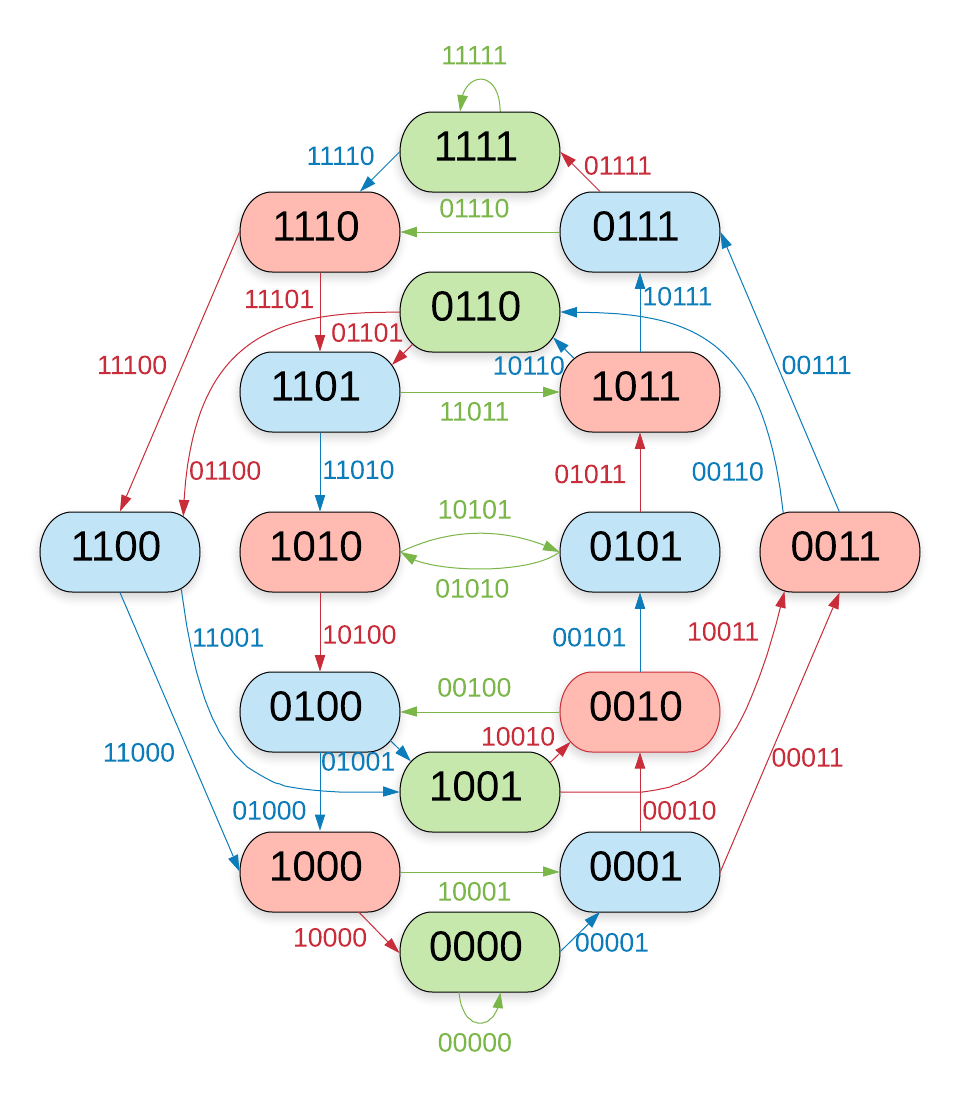}
    \caption{The $4$-dimensional de Bruijn graph. The vertices and edges are coloured red if $\psi = -1$, green if $\psi = 0$ and blue if $\psi = 1$.}
    \label{fig:deBruijn4psi}
\hrulefill\end{figure}

\begin{lemma}
    \begin{align}
        \p{\psi \circ C} \p w = 
        \p{-1}^{\# w + 1} \cdot \psi \p w &&
        \psi \circ R = -\psi 
        \label{eq:psiCpsiR}
    \end{align}
\end{lemma}
\begin{proof}
    \begin{align*}
        \begin{split}
            &\p{\psi \circ C} \p w 
        \\
            \spaceby{\eqref{eq:psi}}{
                {\eqref{eq:xiCxiR}},
                {\eqref{eq:phiCphiR}}
            } 
            &\p{\xi \p{C \p w}}^{\# w} \cdot 
            \phi \p{C \p w} 
        \\
            \by{
                {\eqref{eq:xiCxiR}},
                {\eqref{eq:phiCphiR}}
            }
            &\p{-\xi \p w}^{\# w} \cdot 
            -\phi \p w 
        \\
            \spaceby{\eqref{eq:psi}}{
                {\eqref{eq:xiCxiR}},
                {\eqref{eq:phiCphiR}}
            }
            &\p{-1}^{\# w + 1} \cdot \psi \p w      
        \end{split}
        & &\Bigg | &
        \begin{split}
            &\p{\psi \circ R} \p w 
        \\
            \spaceby{\eqref{eq:psi}}{
                {\eqref{eq:xiCxiR}},
                {\eqref{eq:phiCphiR}}
            } 
            &\p{\xi \p{R \p w}}^{\# w} \cdot 
            \phi \p{R \p w} 
        \\
            \by{
                {\eqref{eq:xiCxiR}},
                {\eqref{eq:phiCphiR}}
            }
            &\p{-\xi \p w}^{\# w} \cdot 
            \p{-1}^{\# w + 1} \cdot \phi \p w
        \\
            \spaceby{\eqref{eq:psi}}{
                {\eqref{eq:xiCxiR}},
                {\eqref{eq:phiCphiR}}
            }
            &-\psi \p w              
        \end{split}
    \end{align*}    
\end{proof}

\begin{lemma} \label{lemma:psi0xi0}
    Let $w \in \abc^*$. Then
    \begin{align}
        \psi \p w = 0 \iff \xi \p w = 0
        \label{eq:psi0xi0}.
    \end{align}       
\end{lemma}
\begin{proof}
    \begin{align*}
        \psi \p w = 0 
        \by[\iff]{\eqref{eq:psi}}
        \xi \p w = 0 \lor \phi \p w = 0
        \by[\iff]{\eqref{eq:phi0xi0}}
        \xi \p w = 0
    \end{align*}
\end{proof}
\begin{remark}
    Lemma~\ref{lemma:psi0xi0} will often be used without explicit reference. Saying that $w$ is colourable should always be understood as $\xi \p w, \psi \p w \ne 0$. To show colourability, of course the easier condition $\xi \p w \ne 0$ will be used. 
\end{remark}

\begin{corll}
    Let $w \in \abc^\even$. Then 
    \begin{align}
        \psi \p w = \phi \p w.
        \label{eq:psiphi}
    \end{align}
\end{corll}
\begin{proof}
    The statement follows directly from lemma~\ref{lemma:psi0xi0} and the definition of $\psi$. 
\end{proof}

\begin{thrm} \label{thrm:psilllpsirrr}
    Let $k, n \in \N, w \in \abc^\N$ such that $k \ge 1$ and 
    \begin{align*}
        \slice w \j {\j+n} \text{ is colourable} \iff
        \j \in \crl{0, k}.
    \end{align*}
    Then 
    \begin{align}
        \psi \p{\slice w 0 n} \ne 
        \psi \p{\slice w k {k+n}}
        \label{eq:psilllpsirrr}.
    \end{align}
\end{thrm}
\begin{proof}
    The statement can be seen as a corollary to lemma~\ref{lemma:philllphirrr}. First let $n$ be even. Then 
    \begin{align*}
        \psi \p{\slice w 0 n} 
        \by{\eqref{eq:psiphi}}
        \phi \p{\slice w 0 n}
        \by[\ne]{\eqref{eq:philllphirrr}}
        \phi \p{\slice w k {k+n}}
        \by{\eqref{eq:psiphi}}
        \psi \p{\slice w k {k+n}}.
    \end{align*}
    Now let $n$ be odd. Then 
    \begin{align*}
        \psi \p{\slice w 0 n} 
        \by{\eqref{eq:psi}}
        &\xi \p{\slice w 0 n} \cdot 
        \phi \p{\slice w 0 n}
    \\
        \by[\ne]{\eqref{eq:philllphirrr}}
        &\xi \p{\slice w k {k+n}} \cdot 
        \phi \p{\slice w k {k+n}}
    \\
        \by{\eqref{eq:psi}}
        &\psi \p{\slice w k {k+n}}.
    \end{align*}
\end{proof}

\begin{remark}
    Theorem~\ref{thrm:psilllpsirrr} states that in an infinite string, a colourable substring always has a different $\psi$"~colour than the preceding colourable substring of the same length -- no matter how many non-colourable substrings there are in between. 
\end{remark}

\begin{figure}
	\begin{huge}
	    \begin{align*}
	            \bm{
	                \blue  0
	                \green 1
	                \red   0
	                \blue  0
	                \green 1
	                \green 1
	                \red   1
	                \blue  1
	                \red   1
	                \blue  0
	                \red   0
	                \green 0
	                \blue  1
	                101... 
	            } 
	    \end{align*}
	\end{huge}
    \caption{The beginning of an infinite string $w \in \abc^\N$. For $k \in \N$ the character on place $k$ has been coloured red if $\psi \p{\slice w k {k+4}} = -1$, green if $\psi \p{\slice w k {k+4}} =  0$ and blue if $\psi \p{\slice w k {k+4}} =  1$. The first character is blue because $\psi \p{\bm{0100}} = 1$; the second one is green because $\psi \p{\bm{1001}} = 0$. The last three characters are printed black because their colour depends on the upcoming characters that are not given here. Ignoring the green characters, always a blue character follows on a red and vice versa. The fact that there are green characters means that this $w$ cannot be a point in a $4$-colourable shift.} 
    \label{fig:shiftpointnoncol}
\hrulefill\end{figure}

\begin{corll} \label{corll:psilpsir}      
    Let $w \in \abc^{\ge 1}$ be such that $l \p w, r \p w$ are colourable. 
    Then $\psi \p{l \p w} \ne \psi \p{r \p w}$. 
\end{corll}
\begin{proof}
    The statement follows from theorem~\ref{thrm:psilllpsirrr} by regarding $w$ as the beginning of an infinite string and setting $k = 1, n = \# w - 1$. 
\end{proof}

\begin{remark}
    Compare corollary ~\ref{corll:psilpsir} with corollary ~\ref{corll:philphir}. While there were some special cases in which $\phi \p{l \p w} \ne \phi \p{r \p w}$ would not hold, $\psi \p{l \p w} \ne \psi \p{r \p w}$ always does. 
\end{remark}

\begin{figure}
    \centering
    \includegraphics[width=\textwidth]{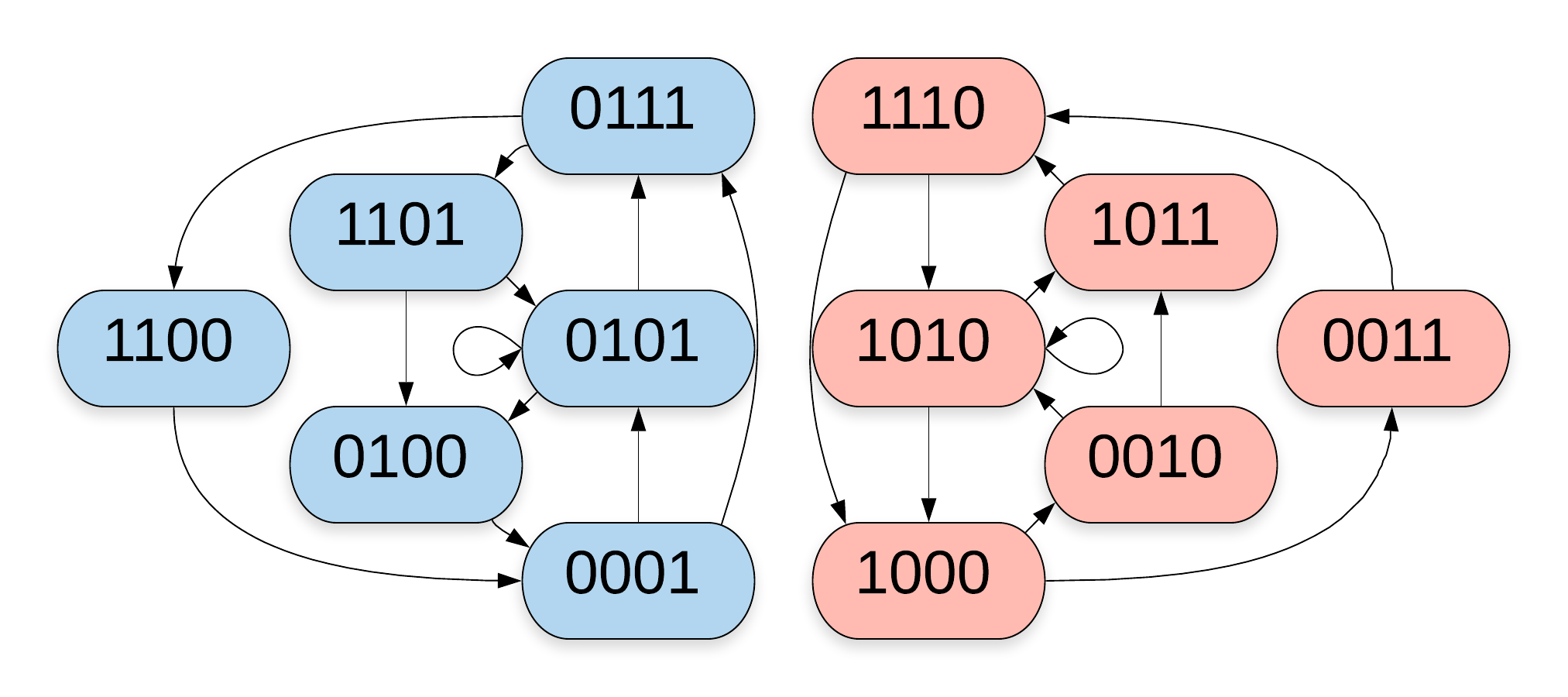}
    \caption{The \nth{2} power graph of the $4$-dimensional de Bruijn graph. The vertices whose $\xi$-value is zero are omitted as are the edges that represent walks which would have passed an omitted edge. The vertices are coloured red if $\xi = -1$ and blue if $\xi = 1$.}
    \label{fig:deBruijn4power}
    \hrulefill
\end{figure}

\begin{remark}
    Consider a de Bruijn graph of which all non-colourable vertices have been removed. Corollary~\ref{corll:psilpsir} states that the colour of the vertices a walk passes alternates between blue and red. For the resulting vertex shift that result will be expressed in corollary~\ref{corll:psialters}.
    
    In the even powers of such a graph there are only walks from vertices of one colour to vertices of the same colour. Figure~\ref{fig:deBruijn4power} shows the \nth{2} power of the $4$-dimensional de Bruijn subgraph without non-colourable edges. 
\end{remark}

\begin{corll} \label{corll:psialtersz}
    Let $k, n \in \N, w \in \abc^\N$ such that ${\slice w 0 n}, {\slice w k {k+n}}$ are colourable. Set $$z = \# \crl{\j < k; \slice w \j {\j + n} \text{ is not colourable}}.$$ Then 
    \begin{align}
        \psi \p{\slice w k {k+n}} = 
        \alter{k-z} \cdot \psi \p{\slice w 0 n}
        \label{eq:psialtersz}.
    \end{align}
\end{corll}
\begin{proof}
    Strong induction on $k$. For $k = 0$ the statement is trivial. Now assume the statement holds whenever $k < K$ for a certain $K \in \N$. Set 
        \begin{align*}
            Z =& \spacemath{\#}{\max{}} \crl{
                \j < K;
                \slice w \j {\j+n} 
                \text{ is not colourable}
            }, 
        \\
            k =& \max \crl{
                \j < K; 
                \slice w \j {\j + n} 
                \text{ is colourable}
            }, 
        \\
            z =& \spacemath{\#}{\max{}} \crl{
                \j < k;
                \slice w \j {\j+n} 
                \text{ is not colourable}
            }.
        \end{align*}
        Note that $Z - z = K - k - 1$, so 
        \begin{align}
            k - z = K - Z - 1
            \label{eq:proof:JZ}.
        \end{align} 
        Hence 
        \begin{align*}
            \psi  \p{\slice w K {K+n}} 
            \by{\eqref{eq:psilllpsirrr}}
            &-\psi \p{\slice w k {k+n}}
        \\
            \by{\eqref{eq:psilllpsirrr}}
            &-\alter{k-z} \cdot 
            \psi \p{\slice w 0 n}
        \\
            \by{\eqref{eq:proof:JZ}}
            &\alter{K-Z} \cdot 
            \psi \p{\slice w 0 n}.
        \end{align*}
\end{proof}

\begin{corll}
    Take $k, n, w, z$ as in corollary~\ref{corll:psialtersz}. 
    
    Let $\psi \p{\slice w 0 n} = \psi \p{\slice w k {k+n}}$. Then 
    \begin{align*}
        k \iseven \iff 
        z \iseven[.]
    \end{align*}
\end{corll}
\begin{proof}
    By corollary~\ref{corll:psialtersz} $\alter{k-z} = 1$. 
\end{proof}

\begin{corll} \label{corll:psialters}
    Let $n \in \N$ and $w$ be a point in an $n$"~colourable shift. Then 
    \begin{align}
        \forall k \in \N \quad 
        \psi \p{w_{[k, k+n)}} = 
        \alter k \cdot \psi \p{w_{[0, n)}}
        \label{eq:psialters}.
    \end{align}
\end{corll}
\begin{proof}
    The statement follows directly from corollary~\ref{corll:psialtersz} by noting that for an $n$"~colourable shift, $z = 0$. 
\end{proof}

\begin{remark}
    Corollary~\ref{corll:psialters} motivates the term \emph{alternating colouring}. 
\end{remark}

\begin{figure}
	\begin{huge}
	    \begin{align*}
	            \bm{
	                \blue 0
	                \red  0
	                \blue 0
	                \red  1
	                \blue 0
	                \red  1
	                \blue 1
	                \red  1
	                \blue 0
	                \red  0
	                \blue 0
	                \red  1
	                \blue 1
	                \red  1
	                \blue 0
	                \red  1
	                \blue .
	                \red  .
	                \blue .
	            }
	    \end{align*}
	\end{huge}
    \caption{The beginning of a point $w$ in a $4$-colourable shift. For $k \in \N$ the character on place $k$ has been coloured blue if $\psi \p{\slice w k {k+4}} = 1$ and red if $\psi \p{\slice w k {k+4}} = -1$. The sequence alternates between blue and red, a visualisation of the alternating colouring.}
    \label{fig:shiftpointcol}
\hrulefill\end{figure}

\begin{corll} \label{corll:closedeven}
    Let $k, n \in \N$ and $w$ be a point in an $n$"~colourable shift such that $\slice w 0 n = \slice w k {k+n}$. Then $k$ is even. 
\end{corll}
\begin{proof}
    By corollary~\ref{corll:psialters} $\alter k = 1$. 
\end{proof}

\begin{remark}
    Corollary~\ref{corll:closedeven} states that any closed walk among the colourable edges of a de Bruijn graph has even length. 
\end{remark}

\begin{corll}
    Let $n \in \N$. A non-empty, $n$"~colourable shift is not mixing. 
\end{corll}
\begin{proof}
    For any natural number there is a larger odd number, but by corollary~\ref{corll:closedeven} there is no walk of odd length from a word back to itself. 
\end{proof}

\section{Counting the non-colourable strings}
\label{section:Numbernoncolourable}

\begin{observ} \label{observ:coliffpalin}
    A string of length $\le 5$ is colourable if and only if it is not a palindrome. 
\end{observ}

\begin{lemma}
    Let $n \ge 4$. Then 
    \begin{align}
        \#\crl{
            w \in \abc^n; \cp \xi m w = 0 \land \xi \p w \ne 0
        } = 
        6 - 2 \cdot \alter n
        \label{eq:hashxm0xn0}
    \end{align}
\end{lemma}
\begin{proof}
    The statement can be seen as a corollary to lemma~\ref{lemma:ximw0w}. Set $$W = \crl{w \in \abc^n; \cp \xi m w = 0}.$$ Then 
    \begin{align*}
        &
        \#\crl{
            w \in \abc^n; 
            \cp \xi m w = 0 \land 
            \xi \p w \ne 0
        }
    \\
        \by[=]{\phantom{\eqref{eq:ximw0w0n11}}}&
        \# \crl{w \in W; \xi \p w \ne 0} 
    \\
        \by[=]{\phantom{\eqref{eq:ximw0w0n11}}}&
        \# \crl{
            w \in W; 
            \cp \xi l w = 0 \land 
            \cp \xi r w \ne 0
        } 
    \\ 
        {}+{} &
        \# \crl{
            w \in W; 
            \cp \xi l w \ne 0 \land 
            \cp \xi r w = 0
        } 
    \\    
        {}+{} & 
        \# \crl{
            w \in W; 
            \cp \xi l w = 
            \cp \xi r w \ne 
            0
        }
    \\ 
        \by[=]{\eqref{eq:ximw0w0n11}} &2 
        \by[+]{\eqref{eq:ximw0w01n1}} 2 
        \by[+]{
            \eqref{eq:ximw0w0Tn1Tn11},
            \eqref{eq:ximw0wCTn101CTn1}
        } 2 \cdot \p{1 - \alter n} 
    \\ 
        \by[=]{\phantom{\eqref{eq:ximw0w0n11}}}
        &6 - 2 \cdot \alter n.
    \end{align*}
\end{proof}

\begin{corll} \label{corll:KnKn2}
    For $n \in \N$ set $K_n = \# \crl{w \in \abc^n; \psi \p w = 0}$. Let $n \ge 4$. Then 
    \begin{align}
        K_n = 
        4 \cdot K_{n-2} - 6 + 2 \cdot \alter n.
        \label{eq:KnKn2}
    \end{align}
\end{corll}
\begin{proof}
    \begin{align*}
        K_n 
        \by{\eqref{eq:psi0xi0}}
        \# \crl{w \in \abc^n; \xi \p w = 0}
        \by{\eqref{eq:xi0xim0}}& 
        \# \crl{  
            w \in \abc^n; 
            \cp \xi m w = 0
        }
    \\*
        {}-{}&
        \# \crl{
            w \in \abc^n; 
            \cp \xi m w = 0 \land 
            \xi \p w \ne 0
        }
    \\ 
        \by{\eqref{eq:hashxm0xn0}}&
        4 \cdot K_{n-2} - 6 + 2 \cdot \alter n.
    \end{align*}    
\end{proof}

\begin{defn}
    For $n \in \N$ the \empar{Jacobsthal number} $J_n$ is defined as 
    \begin{align}
        J_n = \frac{2^n - \alter n}{3}
        \label{eq:Jn}.
    \end{align}
\end{defn}

\begin{table}
    \centering
    \begin{tabular}{c|rrrrrrrrrrrr}
        $n$ & 0 & 1 & 2 & 3 & 4 & 5 & 6 & 7 & 8 & 9 & 10 & 11 
    \\ \hline
        $J_n$ & 0 & 1 & 1 & 3 & 5 & 11 & 21 & 43 & 85 & 171 & 341 & 683
    \end{tabular}
    \caption{The first Jacobsthal numbers}
    \label{tab:J}
\hrulefill\end{table}

\begin{fact} \label{fact:Jn12}
    Let $n \in \N$. 
    \begin{align}
        J_{n+1} &= 
        2 \cdot J_{n} + \alter {n} 
        \label{eq:Jn1}
    \\       
        J_{n+2} &=  
        J_{n+1} + 2 \cdot J_n
        \label{eq:Jn2}
    \end{align}
\end{fact}

\begin{remark}
    The definition of the Jacobsthal numbers, the values in table~\ref{tab:J} and fact~\ref{fact:Jn12} have been taken from \cite{wiki:jacobsthal}.
\end{remark}

\begin{thrm} \label{thrm:Kn2}
    Let $n \in \N$. 
    \begin{align}
        K_{n+2} = 2 \cdot \p{J_n + 1}. 
        \label{eq:Kn2}
    \end{align}
\end{thrm}
\begin{proof}
    By induction on $n$. For $n \le 1$ there are only two cases to consider. Now assume the statement holds for a certain $n$. 
    \begin{align*}
        K_{n+4}
        \by{\eqref{eq:KnKn2}}&
        4 \cdot K_{n+2} - 6 + 
        2 \cdot \alter n
    \\ 
        \by{\eqref{eq:Kn2}}&
        8 \cdot \p{J_n + 1} - 6 + 
        2 \cdot \alter n
    \\ 
        \by{\phantom{\eqref{eq:Kn2}}}&
        2 \cdot \p{
            2 \cdot J_n + \alter n
        } + 
        4 \cdot J_n + 2
    \\  
        \by{\eqref{eq:Jn1}}&
        2 \cdot J_{n+1} + 
        4 \cdot J_n + 2
    \\  
        \by{\eqref{eq:Jn2}}&
        2 \cdot \p{J_{n+2} + 1}
    \end{align*}
\end{proof}

\begin{table}
    \centering
    \begin{tabular}{c|rrrrrrrrrrrr}
        $n$ & 0 & 1 & 2 & 3 & 4 & 5 & 6 & 7 & 8 & 9 & 10 & 11 
    \\ \hline
        $K_n$ & 1 & 2 & 2 & 4 & 4 & 8 & 12 & 24 & 44 & 88 & 172 & 344
    \end{tabular}
    \caption{The number $K_n$ of non-colourable strings in $\abc^n$}
    \label{tab:K}
\hrulefill\end{table}

\begin{corll} \label{corll:Kn}
    Let $n \ge 2$. Then 
    \begin{align}
        n \iseven \implies&
        K_n = \frac{2^n + \spacemath{8}{16}}{6}
        \label{eq:Kneven}
    \\ 
        n \isodd  \implies&
        K_n = \frac{2^n + 16}{6}
        \label{eq:Knodd}.
    \end{align}
\end{corll}
\begin{proof}
    \begin{align*}
        K_n 
        \by{\eqref{eq:Kn2}}
        2 \cdot \p{J_{n-2} + 1} 
        \by{\eqref{eq:Jn}}
        2 \cdot \p{
            \frac{2^{n-2} - \alter n}{3} + 1
        } = 
        \frac{2^n + 12 - 4 \cdot \alter n}{6}    
    \end{align*}
\end{proof}

\begin{corll} \label{corll:Kn1Kneven}
    Let $n$ be even. Then 
    \begin{align}
        K_{n+1} = 2 \cdot K_n
        \label{eq:Kn1Kneven}.
    \end{align}
\end{corll}
\begin{proof}
    \begin{align*}
        K_{n+1} 
        \by{\eqref{eq:Knodd}}
        \frac{2^{n+1} + 16}{6}
        = 
        2 \cdot \frac{2^n + 8}{6}
        \by{\eqref{eq:Kneven}}
        2 \cdot K_n
    \end{align*}
\end{proof}

\begin{corll} \label{corll:Kn2n}
    \begin{align*}
        \lim_{n \to \infty} 
        \frac{K_n}{2^n} = 
        \frac{1}{6}
    \end{align*}    
\end{corll}
\begin{proof}
    This follows immediately from corollary ~\ref{corll:Kn}. 
\end{proof}

\begin{remark}
    Since there are $2^n$ strings of length $n$, corollary ~\ref{corll:Kn2n} tells that in the limit 1 out of 6 strings is not colourable. 
\end{remark}

\begin{corll}
    Let $n \ge 2$. Then 
    \begin{align*}
        K_{n+4} = 5 \cdot K_{n+2} - 4 \cdot K_n.
    \end{align*}
\end{corll}
\begin{proof}
    Let $c \in \N$. The statement follows from corollary ~\ref{corll:Kn} and the observation that 
    \begin{align*}
        5 \cdot \frac{2^{n+2} + c}{6} - 
        4 \cdot \frac{2^{n  } + c}{6} = 
        \frac{5 \cdot 2^{n+2} - 2^{n+2} + c}{6} = 
        \frac{2^{n+4} + c}{6}.         
    \end{align*}
\end{proof}

\section{Colourability sources and sinks}

\begin{defn}
    Let $w \in \abc^*$. If $w \zero, w \one$ are not colourable, $w$ is called a \empar{colourability sink}; if $\zero w, \one w$ are not colourable, $w$ is called a \empar{colourability source}.
    $\upepsilon$ is called the \empar[trivial colourability sink and source]{trivial} colourability sink and source. 
\end{defn}

\begin{observ} \label{observ:reversesource}
    The reverse of a colourability source is a colourability sink and vice versa. 
\end{observ}


\begin{corll}
    Let $n \in \N$. There are equally many colourability sinks as sources in $\abc^n$. 
\end{corll}
\begin{proof}
    If $w \in \abc^n$ is a colourability sink, $R \p w$ is a colourability source and vice versa. Since $R$ is bijective, the statement follows.
\end{proof}

\begin{lemma} \label{lemma:xi0wxi1wsink}
    If $w \in \abc^*$ is a colourability sink $\zero w, \one w$ are colourable. If $w$ is a colourability source $w \zero, w \one$ are colourable. 
\end{lemma}
\begin{proof}
    Assume the contrary, say $\xi \p{w \zero} = \xi \p{\zero w} = 0$. Then by lemma~\ref{lemma:xilr0} $\zero w \zero \in \crlzero *$, so $w \in \crlzero *$. Hence $\xi \p{\one w}, \xi \p{w \one} \by[\ne]{\eqref{eq:xi0i1j}, \eqref{eq:xi1j0i}} 0,$ a contradiction. Similarly if $\xi \p{w \one} = \xi \p{\one w} = 0$ one gets that $\xi \p{\zero w}, \xi \p{w \zero} \ne 0$. 
\end{proof} 

\begin{corll} \label{corll:sinksource}
    The only string that is both a colourability sink and a colourability source is $\upepsilon$. 
\end{corll}
\begin{proof}
    Any non-trivial colourability sink is by lemma~\ref{lemma:xi0wxi1wsink} not a colourability source. 
\end{proof}

\begin{remark}
    Corollary ~\ref{corll:sinksource} means that removing the non-colourable edges of a de Bruijn graph does not leave isolated vertices.
\end{remark}

\begin{lemma} \label{lemma:sinkcol}
    Any non-trivial colourability sink is colourable. 
\end{lemma}
\begin{proof}
    Let $w \in \abc^*$ not be colourable. 
    \begin{align*}
        \xi \p w = 
        \xi \p{w \zero} = 
        0 
        \spaceby[\implies]{\eqref{eq:xi0}}{\eqref{eq:xilr0xi0}}
        &\xi \p w = 
        \xi \p{r \p w \zero} = 
        \xi \p{w \zero} = 
        0 
    \\
        \by[\implies]{\eqref{eq:xilr0xi0}}
        &w \zero \in \crlzero * \cup \crlone *
    \\
        \spaceby[\implies]{}{\eqref{eq:xilr0xi0}}
        &w \in \crlzero *   
    \end{align*}
    but similarly
    \begin{align*}
        \xi \p w = 
        \xi \p{w \one} = 
        0 
        \implies 
        w \in \crlone *,
    \end{align*}
    so $w = \upepsilon$. 
\end{proof}

\begin{corll} \label{corll:sourcecol}
    Any non-trivial colourability source is colourable. 
\end{corll}
\begin{proof}
    Since the reverse of a colourability source is a colourability sink, the statement follows directly from lemmata \ref{lemma:xiCxiR} and \ref{lemma:sinkcol}.
\end{proof}

\begin{lemma} \label{lemma:sinksourcexipsi}
    If $w \in \abc*$ is a colourability sink,
    \begin{align}
        \xi \p{\zero w} &= 
        \xi \p{\one w} = 
        \xi \p w 
        \label{eq:xi0sinkxi1sink}
    \\ 
        \phi \p{\zero w} &= 
        \phi \p{\one w} = 
        \phi \p w
        \label{eq:phi0sinkphi1sink}
    \\ 
        \psi \p{\zero w} &= 
        \psi \p{\one w} = 
        \xi \p w \cdot \psi \p w
        \label{eq:psi0sinkpsi1sink}.
    \end{align}
    If $w$ is a colourability source,
    \begin{align}
        \xi \p{w \zero} &= 
        \xi \p{w \one} = 
        \xi \p w 
        \label{eq:xisource0xisource1}
    \\ 
        \phi \p{w \zero} &= 
        \phi \p{w \one} = 
        -\phi \p w
        \label{eq:phisource0phisource1}
    \\ 
        \psi \p{w \zero} &= 
        \psi \p{w \one} = 
        -\xi \p w \cdot \psi \p w
        \label{eq:psisource0psisource1}.
    \end{align}
\end{lemma}
\begin{proof}
    Consider lemma~\ref{lemma:xi0wxi1wsink},  lemma~\ref{lemma:sinkcol} and corollary~\ref{corll:sourcecol}. 
    For \eqref{eq:xi0sinkxi1sink} and \eqref{eq:xisource0xisource1} combine them with lemma~\ref{lemma:xixil}, for \eqref{eq:phi0sinkphi1sink} and \eqref{eq:phisource0phisource1} with lemma~\ref{lemma:phiphir}.  
    The following establishes \eqref{eq:psi0sinkpsi1sink}:
    \begin{align*}
        \begin{split}
            \psi \p{\zero w} &= 
            \p{\xi \p{\zero w}}^{\# w + 1} \cdot 
            \phi \p{\zero w} 
            \by{\eqref{eq:xi0sinkxi1sink}}
        \\ 
            \psi \p{\one w} &= 
            \p{\xi \p{\one w}}^{\# w + 1} \cdot 
            \phi \p{\one w} 
            \by{\eqref{eq:phi0sinkphi1sink}}
        \end{split}
        \p{\xi \p w}^{\# w + 1} \cdot 
        \phi \p w 
        = 
        \xi  \p w \cdot 
        \psi \p w
    \end{align*}
    and the following \eqref{eq:psisource0psisource1}: 
    \begin{align*}
        \begin{split}
            \psi \p{w \zero} 
            &= 
            \p{\xi \p{w \zero}}^{\# w + 1} \cdot 
            \phi \p{w \zero} 
            \by{\eqref{eq:xisource0xisource1}}
        \\ 
            \psi \p{w \one} 
            &= 
            \p{\xi \p{w \one}}^{\# w + 1} \cdot 
            \phi \p{w \one} 
            \by{\eqref{eq:phisource0phisource1}}
        \end{split}  
        \p{\xi \p w}^{\# w + 1} \cdot 
        \p{-\phi \p w}
        = 
        -\xi \p w \cdot 
        \psi \p w.       
    \end{align*}
\end{proof}

\begin{thrm} \label{thrm:psiw0psiw1}
    Let $w \in \abc^*$. Then 
    \begin{align*}
        w \zero, w \one 
        \text{ are colourable} 
        &\implies
        \psi \p{w \zero} = 
        \psi \p{w \one}
    \\ 
        \zero w, \one w
        \text{ are colourable} 
        &\implies
        \psi \p{\zero w} = 
        \psi \p{\one w}
    \end{align*}
\end{thrm}
\begin{proof}
    If $w$ is a colourability sink or source, the statement is proven in lemma~\ref{lemma:sinksourcexipsi}. Otherwise it follows from corollary~\ref{corll:psilpsir}.
\end{proof}

\begin{figure}
    \centering
    \includegraphics[height=\graphicheight]{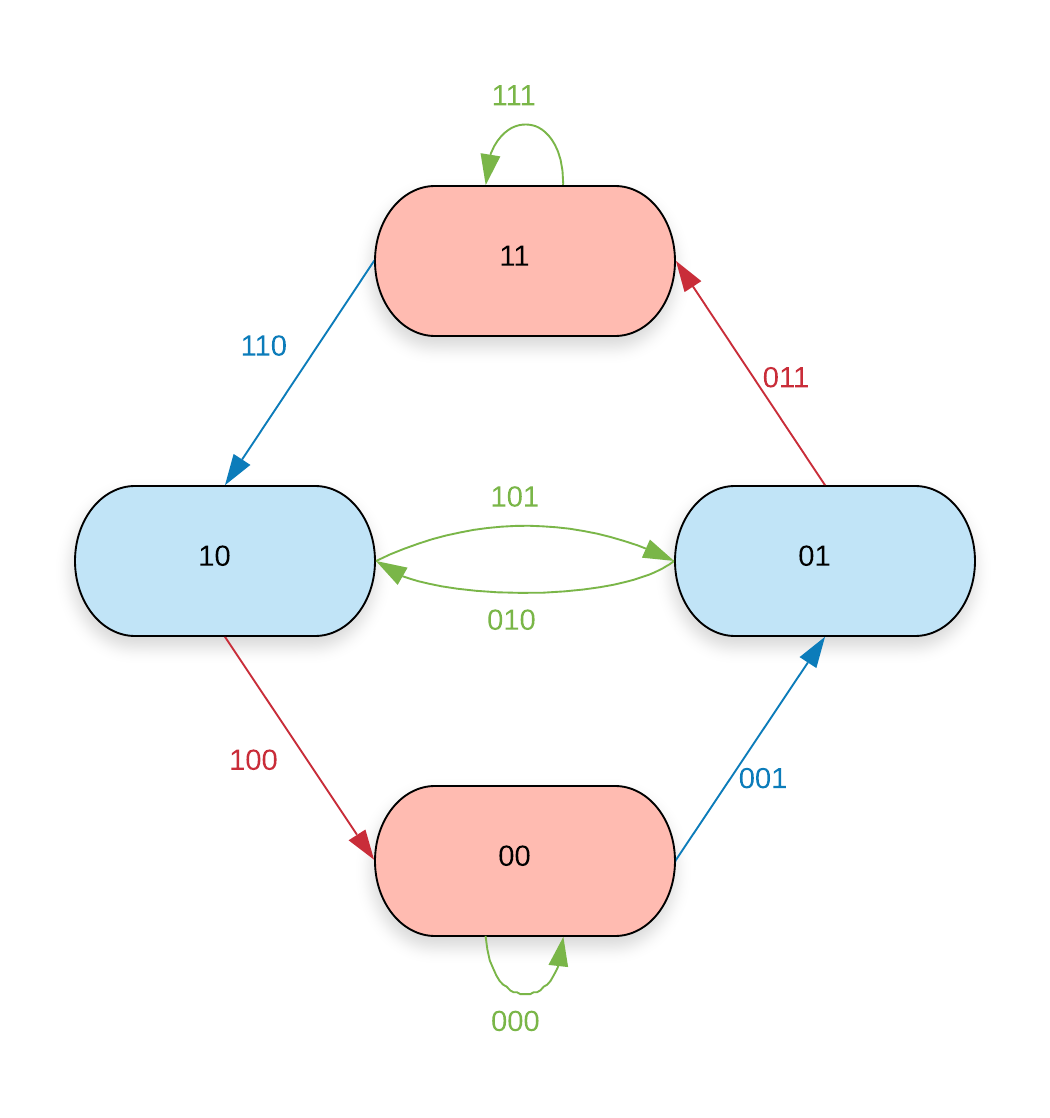}
    \caption{The $2$-dimensional de Bruijn graph. As in figure~\ref{fig:deBruijn2psi} the edges are coloured according to the value $\psi$ assigns to them. The colour of the vertices coincides with the colour of the incoming edges and differs from the colour of the outgoing ones.}
    \label{fig:deBruijn2psivertices}
    \hrulefill
\end{figure}
\begin{figure}
    \centering
    \includegraphics[height=\textwidth]{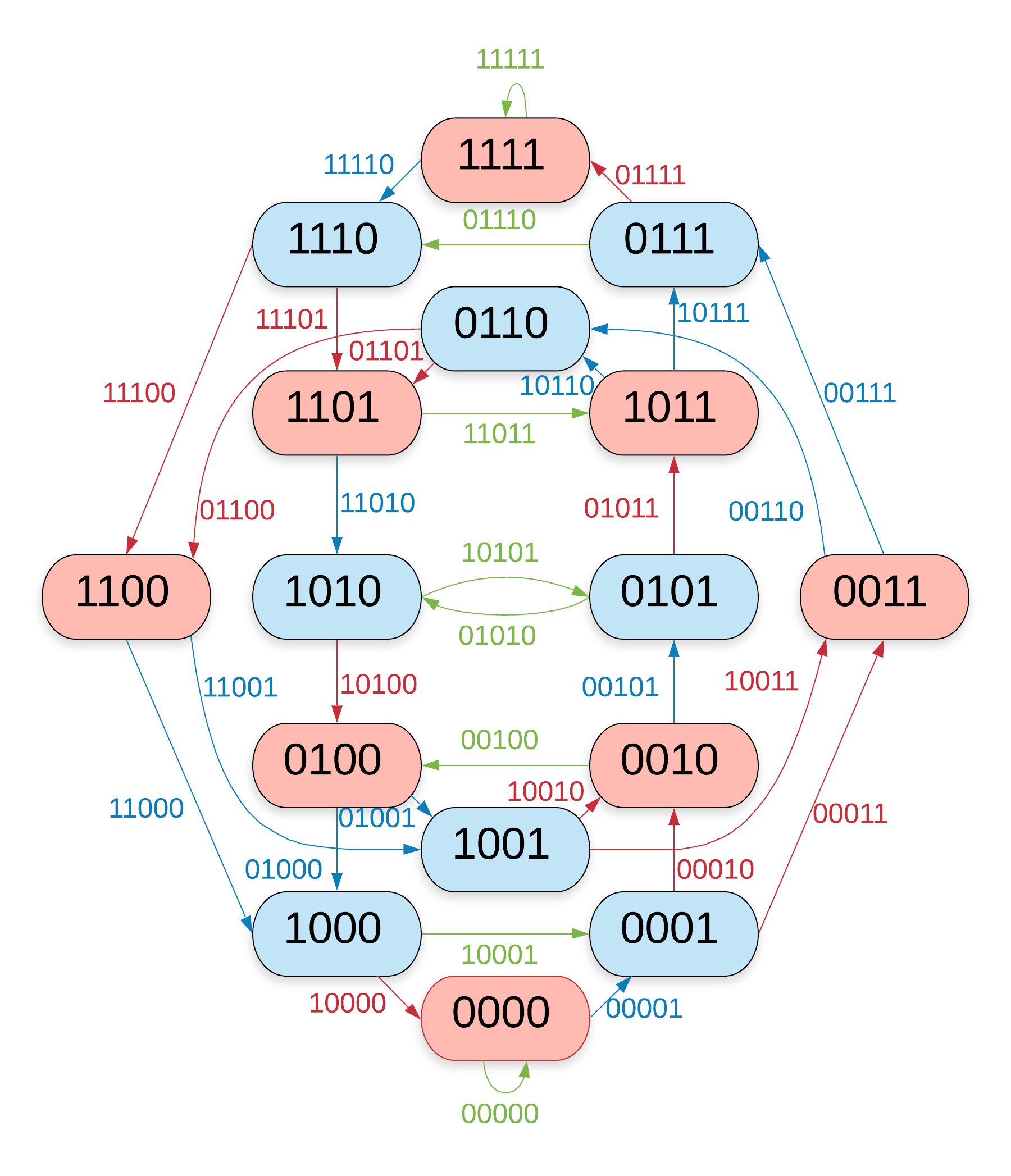}
    \caption{The $4$-dimensional de Bruijn graph. As in figure~\ref{fig:deBruijn4psi} the edges are coloured according to the value $\psi$ assigns to them. The colour of the vertices coincides with the colour of the incoming edges and differs from the colour of the outgoing ones.}
    \label{fig:deBruijn4psivertices}
    \hrulefill
\end{figure}

\begin{remark}
    Theorem~\ref{thrm:psiw0psiw1} makes it possible to define a function that assigns to the vertices of a de Bruijn graph a colour that corresponds to the colour of the incoming edges while differing from the colour of the outgoing ones. Lemma~\ref{lemma:phiphir} and corollary~\ref{corll:phiodd} show that for odd-dimensional de Bruijn graphs that function is just $\phi$. Figures~\ref{fig:deBruijn2psivertices} and \ref{fig:deBruijn4psivertices} give examples for even-dimensional de Bruijn graphs. 
\end{remark}

\begin{observ}
    In $\abc^{\le 4}$ there are no non-trivial colourability sinks or sources. 
\end{observ}

\begin{thrm} \label{thrm:Sn}
    For $n \in \N$ let $S_n$ be the number of colourability sinks in $\abc^n$. If $n \ge 3$, $$S_n = K_n - 4.$$ 
\end{thrm}
\begin{proof}
    Let $n \ge 3$ and $w \in \abc^n$ be a colourability sink. Then by lemma~\ref{lemma:xi0xim0} $\cp \xi r w = 0$. 
    First, let $n$ be odd. By lemma~\ref{lemma:ximw0w} 
    \begin{align*}
        \xi \p{w \zero} = 
        \xi \p{w \one} = 
        0 \iff
        \cp \xi r w = 0 \land 
        w \notin \crl{
            \zero^n, 
            \zero \one^{n-1},
            \one \zero^{n-1},
            \one^n
        },
    \end{align*}
    so 
    \begin{align*}
        S_n = 
        \# \crl{
            w \in \abc^n;
            \cp \xi r w = 0
        } - 4 = 
        2 \cdot K_{n-1} - 4 
        \by{\eqref{eq:Kn1Kneven}}
        K_n - 4.
    \end{align*}
    Now let $n$ be even. Then by lemma~\ref{lemma:ximw0w} 
    \begin{align*}
        &\xi \p{w \zero} = 
        \xi \p{w \one} = 
        0 
    \\
        \iff
        &\cp \xi r w = 0 \land 
        w \notin \crl{
            \zero^n, 
            \zero T^{n-1},
            T^n,
            \zero \one^{n-1},
            \one \zero^{n-1},
            CT^n,
            \one CT^{n-1},
            \one^n
        },      
    \end{align*}
    so 
    \begin{align*}
        S_n 
        \by{\phantom{\eqref{eq:Kneven}}}&
        \# \crl{
            w \in \abc^n;
            \cp \xi r w = 0
        } - 8 
    \\
        \by{\phantom{\eqref{eq:Kneven}}}&
        2 \cdot K_{n-1} - 8 
    \\
        \by{\eqref{eq:Knodd}}&
        2 \cdot \frac{2^{n-1} + 16}{6} - 8 
    \\
        \by{\phantom{\eqref{eq:Kneven}}}& 
        \frac{2^n + 8}{6} - 4 
    \\
        \by{\eqref{eq:Kneven}}  &
        K_n - 4. 
    \end{align*}
\end{proof}

\begin{table}
    \centering
    \begin{tabular}{c|rrrrrrrrrrrr}
        $n$ & 0 & 1 & 2 & 3 & 4 & 5 & 6 & 7 & 8 & 9 & 10 & 11 
    \\ \hline
        $S_n$ & 1 & 0 & 0 & 0 & 0 & 4 & 8 & 20 & 40 & 84 & 168 & 340
    \end{tabular}
    \caption{The number $S_n$ of colourability sinks in $\abc^n$. There are equally many colourability sources as sinks.}
    \label{tab:S}
\hrulefill\end{table}

\begin{remark}
    Theorem~\ref{thrm:Sn} is a little bit more surprising than it looks at first glance: Due to the definition of a colourability sink focusing on $w \zero$ and $w \one$, $S_n$ rather says something about $\abc^{n+1}$ than about $\abc^n$. $K_n$ in contrast should be seen as an information about $\abc^n$. 
\end{remark}
%
%
        \printbibliography

@article{46debruijn,
    author  = "Nicolaas Govert {de Bruijn}",
    sortname= "Bruijn",
    title   = "A combinatorial problem",
    journal = "Proceedings of the Section of Sciences of the Koninklijke Nederlandse Akademie van Wetenschappen te Amsterdam",
    language = "English",
    volume  = "49",
    number  = "7",
    pages   = "758--764",
    year    = "1946"
}

@online{wiki:jacobsthal,
    author = "{Wikipedia contributors}",
    title = "Jacobsthal number", 
    organization = "{Wikipedia}{,} The Free Encyclopedia",
    year = "2020",
    url = "https://en.wikipedia.org/w/index.php?title=Jacobsthal_number&oldid=954813000",
    urldate = "2020-06-03"
}

@thesis{20garbe,
    author = "Jonathan Garbe",
    title = "Observations on de Bruijn Graphs",
    school = "Lund University",
    year = "2020",
    address = "Lund, Sweden",
    url = "https://www.lunduniversity.lu.se/lup/publication/9016243"
}

@online{oeis,
	author = "{OEIS Community}",
	title = "The On-Line Encyclopedia of Integer Sequences",
	year = "2024", 
	url = "https://oeis.org/",
	urldate = "2024-10-26"
}
\end{document}